\documentclass[10pt]{amsart}
\usepackage[english]{babel}
\usepackage{amssymb}
\usepackage{amsmath}
\usepackage{lscape}

\newcommand{\can}{\overline{\phantom{x}}}

\newtheorem{dummy}{Dummy}

\newtheorem{lemma}[dummy]{Lemma}
\newtheorem{theorem}[dummy]{Theorem}

\newtheorem{corollary}[dummy]{Corollary}

\theoremstyle{definition}

\newtheorem{example}[dummy]{Example}

\newtheorem{remark}[dummy]{Remark}

\newcommand{\ignore}[1]{}

\author{D. Thompson}
\author{S. Pumpl\"un}

\email{thompson.danjames@gmail.com; susanne.pumpluen@nottingham.ac.uk}
\address{School of Mathematical Sciences\\
University of Nottingham\\ University Park\\ Nottingham NG7 2RD\\
United Kingdom }



\begin{document}

\title[Division algebras and MRD codes from skew polynomials]
{Division algebras and MRD codes from skew polynomials}

\maketitle

\begin{abstract}
Let $D$ be a  division algebra, finite-dimensional over its center, and $R=D[t;\sigma,\delta]$ a skew polynomial ring.

Using skew polynomials $f\in R$, we construct division algebras and maximum rank distance codes consisting of matrices with entries in a noncommutative division algebra or field. These include Jha Johnson semifields, and the classes of classical and twisted Gabidulin codes constructed by Sheekey.
\end{abstract}

%
%

\section{Introduction}

Rank distance codes are important both in
coding theory and cryptography. One of the best known maximum rank distance (MRD) codes is probably the Gabidulin code \cite{Gabidulin} which was mentioned already by Delsarte \cite{D}. In coding theory, MRD codes are well suited to correct errors \cite{BM85, R91}. In cryptography, they are used to design  public-key cryptosystems, see for instance \cite{GPT91, FL06}.

 MRD codes over general (non-finite) fields, in particular number
fields, were already studied in \cite{ALR} and later touched on in \cite{Sheekey}.
Rank-metric codes over both cyclic and more general Galois extensions were considered in \cite{R91, Roth96, ACLN}.
Although rank metric codes have been also constructed over finite principal ideal rings \cite{KM19} and
discretely valued rings \cite{EHNS}, to our knowledge they have not yet been studied over noncommutative rings. In this paper
we also consider MRD codes in $M_k(B)$, where $B$ is a noncommutative division algebra.

We  construct these MRD codes using skew polynomials.  Skew polynomials have been successfully used in constructions of both division algebras (mostly semifields) and linear codes \cite{ALR,  BU, BL, HGL, P66, P18.0, P18.1}, in particular building space-time block codes (STBCs) \cite{P18.2} and MRD codes \cite{ Sheekey16,  Sheekey}.

Our codes can be seen as generalizations of both the classical and twisted Gabidulin codes in \cite{Gabidulin}, resp., \cite{Sheekey16}.  We put Sheekey's construction \cite{Sheekey} in a broader context which helps to understand it better, and potentially allows other  ways to generalize MRD coding using skew polynomials.
The drawback is that rather early on we have to rigorously restrict the choice of the polynomials $f$ we can employ, and that the construction remains rather theoretical.

Sheekey  \cite{Sheekey} only considers skew polynomials $f\in K[t;\sigma]$  with coefficients in cyclic Galois field extensions for his construction and limits himself to the case that  the minimal central left multiple of $f$ has  maximal degree. He misses out on codes (with matrix entries both in a noncommutative division algebra, and with entries in fields) and algebras that can be obtained by employing skew polynomials with coefficients in a noncommutative division algebra. He also misses out on constructions using $f\in D[t;\delta]$.
 We construct both new division algebras and MRD codes with entries in a noncommutative division algebra, and with entries in fields.

The first five Sections of the paper contain the preliminaries (Section 1) and theoretical background needed to obtain the main results (Sections 2 to 5).
Let $D$ be a division algebra of degree $d$ over its center,
 and $f\in R=D[t;\sigma,\delta]$ a monic irreducible skew polynomial with a bound that lies  in the center $C(R)$ of $R$.

While developing the theory, we point out how the choice of $D$ and the polynomial $f$ has to be  restricted in order to construct both division algebras and  MRD codes out of $f$,  a scalar $\nu\in D$ and a suitable $\rho\in {\rm Aut}(D)$.

Apart from Section \ref{sec:elements in N_general}, we fix the following general assumptions unless specified otherwise:
 $R=D[t;\sigma]$, where $\sigma$ is an automorphism of $D$ of finite order $n$ modulo inner automorphisms, i.e. $\sigma^n=i_u$ for some inner automorphism $i_u(z)=uzu^{-1}$, and $F=C\cap {\rm Fix}(\sigma)$.
 Choose $\rho\in {\rm Aut}(D)$, such that $F/F'$ with $F'={\rm Fix}(\rho)\cap F$ is finite-dimensional. Let $\nu\in D^\times$.

Let $f\in R$ be monic and irreducible of degree $m>1$, and $h$ the minimal central left multiple of $f$, so that
 $ R/Rh \cong M_k(B)$ for some division algebra $B$ (Theorem \ref{thm:main2}). Let $l<k$ be a positive integer.
 Define
$S_{n,m,l}(\nu,\rho, f)=\lbrace a+Rh\,|\, a\in P\rbrace\subset R/Rh$
with the set
$P=\lbrace d_0+d_1t+\dots+d_{lm-1}t^{lm-1}+\nu\rho(d_0)t^{lm}\,|\, d_i\in D\rbrace.$
Let $L_a:R/Rf\to R/Rf$ be the left multiplication map  $L_a(b+Rf)=ab+Rf$.  We have well-defined maps
$ S_{n,m,l}(\nu,\rho, f) \longrightarrow {\rm End}_B(R/Rf)\longrightarrow M_k(B), a\mapsto L_a \mapsto M_a,$
 where $M_a$ is the matrix representing $L_a$
with respect to a right $B$-basis of $R/Rf$. The image $\mathcal{C}_{n,m,l}=\lbrace M_a\mid a\in S_{n,m,l}(\nu,\rho,f)\rbrace$
 of $S_{n,m,l}(\nu,\rho,h)$ in $M_{k}(B)$ is an $F'$-linear  rank metric code.
If $\mathcal{C}_{n,m,l}$ has distance $d_{\mathcal{C}}=k-l+1$, then $\mathcal{C}_{n,m,l}$ is called
a  \textit{maximum rank distance code} in $M_k(B)$.
We will usually deal with the case that $deg (h)=dmn$, so that $B$ is a field.

 The  most general results are contained in Section 6: If $P$ does not contain a polynomial of degree $lm$, whose irreducible factors are all similar to $f$, then $\mathcal{C}_{n,m,l}$ is an $F'$-linear  MRD code in $M_k(B)$ with minimum distance $k-l+1$ (Theorem \ref{thm:MRD}).

 Furthermore, let $D=(E/C,\gamma,a)$ be a cyclic division algebra
such that
$\sigma|_E\in {\rm Aut}(E)$ and $\gamma\circ\sigma|_E=\sigma|_E\circ\gamma$, and
 $\sigma^n(z)=u^{-1}zu$ for some $u\in E$. Let $f(t)=\sum_{i=0}^ma_it^i\in E[t;\sigma]$  be
 a monic irreducible polynomial of degree $m$, 
 such that
 ${\rm deg}(h)=dmn$, and such that all monic $f_i$ similar to $f$  lie in $E[t;\sigma]$. Then the algebra
$S_{n,m,1}(\nu,\rho, f)$  is a division algebra, if one of the following holds:
 (i)  $\nu\not\in E$ and $\rho|_E\in {\rm Aut}(E)$;
(ii) $\nu\in E^\times$  and $\rho|_E\in {\rm Aut}(E)$,
such that
$N_{E/F'}(a_0) N_{E/F'}(\nu)\neq 1$ (Theorem  \ref{thm:division2}).
 MRD codes are canonically obtained from the matrices representing the left multiplication of these division algebras.

In Section 7, the nuclei of the algebras and codes are investigated.
We give some examples of algebras obtained from our construction employing $f(t)=t^n-\theta\in K[t;\sigma]$ in Section 8.

We conclude with a brief look at the constructions using a differential polynomial $f\in D[t;\delta]$, where the center of $D$ is a field of characteristic $p$, in Section \ref{sec:elements in N_general}.

The fact that we are using $f\in D[t;\sigma]$, respectively $f\in D[t;\gamma]$, means we have a larger choice of skew polynomials to build codes that Sheekey does, who
only considers $f$ with coefficients in a cyclic field extension.

This work is part of the second author's PhD thesis \cite{DT2020}.

%
%

\section{Preliminaries}

\subsection{Nonassociative algebras}

Let $F$ be a field. We call $A$ an \emph{algebra} over $F$ if there exists an
$F$-bilinear map $A\times A\to A$, $(x,y) \mapsto x \cdot y$, denoted simply by juxtaposition $xy$,
the  \emph{multiplication} of $A$.
An algebra $A$ is called \emph{unital} if there is
an element in $A$, denoted by 1, such that $1x=x1=x$ for all $x\in A$.
We will only consider unital algebras.
 A nonassociative algebra $A\not=0$ is called a {\it division algebra} if for any $a\in A$, $a\not=0$,
the left multiplication  with $a$, $L_a(x)=ax$,  and the right multiplication with $a$, $R_a(x)=xa$, are bijective.
If $A$ is finite-dimensional as an $F$-vector space, then $A$ is a division algebra if and only if $A$ has no zero divisors.
 The {\it left nucleus} of $A$ is defined as ${\rm
Nuc}_l(A) = \{ x \in A \, \vert \, [x, A, A]  = 0 \}$, the {\it
middle nucleus} of $A$ is ${\rm Nuc}_m(A) = \{ x \in A \, \vert \, [A, x, A]  = 0 \}$ and  the {\it right nucleus} of $A$ is
${\rm Nuc}_r(A) = \{ x \in A \, \vert \, [A,A, x]  = 0 \}$, where $[x, y, z] =
(xy) z - x (yz)$ is the {\it associator}. ${\rm Nuc}_l(A)$, ${\rm Nuc}_m(A)$, and ${\rm Nuc}_r(A)$ are associative
subalgebras of $A$. Their intersection
 ${\rm Nuc}(A) = \{ x \in A \, \vert \, [x, A, A] = [A, x, A] = [A,A, x] = 0 \}$ is the {\it nucleus} of $A$.
${\rm Nuc}(A)$ is an associative subalgebra of $A$,
and $x(yz) = (xy) z$ whenever one of the elements $x, y, z$ is in
${\rm Nuc}(A)$. The
 {\it center} of $A$ is ${\rm C}(A)=\{x\in \text{Nuc}(A)\,|\, xy=yx \text{ for all }y\in A\}$.

Let $A$ be a finite-dimensional central simple associative algebra over $F$ of degree $d$ and let $\overline{F}$ denote the algebraic closure of $F$.
Then $A\otimes_F \overline{F} \cong M_d(\overline{F})$, so that we can fix an embedding $A \longrightarrow M_d(\overline{F})$ and view every $a\in A$ as a matrix in $M_d(\overline{F})$. The characteristic polynomial
$$m_a(X)=X^d-s_1(a)X^{d-1}+s_2(a)X^{d-2}-\dots +(-1)^ds_d(a)$$
of $a\in A$ has coefficients in $F$ and is independent of the choice of the embedding. The coefficient $N_A(a)=s_d(a)$ is called the \emph{reduced norm of} $a$ \cite{KMRT}.
Let $K/F$ be a cyclic Galois extension of degree $d$ with Galois group ${\rm Gal}(K/F)=\langle \gamma \rangle$ and norm $N_{K/F}$. Let $c\in F^\times$.
An \emph{associative cyclic algebra} $(K/F,\gamma,c)$ of degree $d$ over $F$ is a $d$-dimensional $K$-vector space
\[
(K/F,\gamma,c)=K \oplus eK \oplus e^2 K\oplus\dots \oplus e^{d-1}K,
\]
with multiplication given by the relations
$\label{eq:rule}
e^d=c,~le=e\sigma(l),
$
for all $l\in K$.  $(K/F,\gamma,c)$ is a division algebra
 for all $c\in F^\times$, such that
  $c^s\not\in N_{K/F}(K^\times)$ for all $s$ which are prime divisors of $d$, $1\leq s\leq d-1$.

\subsection{MRD-codes}

Let $K$ be a field. A \emph{code} is a set  of matrices $\mathcal{C}\subset M_{n, m}(K)$.
Let $L\subset K$ be a subfield, then $\mathcal{C}$ is \emph{$L$-linear} if $\mathcal{C}$ is a vector space over $L$.
A \emph{rank metric code} is a code $\mathcal{C}\subset M_{n, m}(K)$
 equipped with the rank distance function $d(X,Y)={\rm rank}(X-Y).$ Define the \textit{minimum distance} of a rank metric code $\mathcal{C}$ as $$d_{\mathcal{C}}=\text{min}\lbrace d(X,Y)\mid  X,Y\in \mathcal{C}, X\neq Y\rbrace.$$
 An $L$-linear rank metric code $\mathcal{C}$ satisfies the Singleton-like bound
  $$\text{ dim}_{L}(\mathcal{C})\leq n(m-d_{\mathcal{C}}+1)[K:L],$$
where $\text{ dim}_L(\mathcal{C})$ is the dimension of the $L$-vector space $\mathcal{C}$ \cite[Proposition 6]{ALR}.

An $L$-linear rank metric code attaining the Singleton-like bound is called a  \textit{maximum rank distance code} or \textit{MRD-code} (for MRD-codes over cyclic field extensions see \cite{ALR}).

If now $B$ is a not necessarily commutative division algebra then more generally, we again define a \emph{code} as a set  of matrices $\mathcal{C}\subset M_{n, m}(B)$.
Let $B'\subset B$ be a subalgebra, then $\mathcal{C}$ is \emph{$B'$-linear} (or simply \emph{linear}), if $\mathcal{C}$ is a right $B'$-module.

A \emph{rank metric code} $\mathcal{C}\subset M_{n, m}(B)$ is a code  together with the distance function
$$d(X,Y)=\text{colrank}(X-Y),$$
for all $X,Y\in M_{n,m}(B)$, where ${\rm colrank}$ is the column rank of $A$ (the rank of the right $B$-module generated by the columns of $A$).
A matrix in $M_{n,m}(B)$ has column rank at most $m$; any matrix which attains this bound is said to have attained \emph{full column rank}.
The \emph{minimum distance} of a  rank metric code $\mathcal{C}\subset M_{n,m}(B)$ is defined as
$$d_{\mathcal{C}}=\text{min}\lbrace d(X,Y)\mid  X,Y\in \mathcal{C}, X\neq Y\rbrace.$$
To our knowledge, such codes $\mathcal{C}\subset M_{n, m}(B)$ have not previously been considered in the literature.

\subsection{Skew polynomial rings} \label{subsec:pre}

 In the following let $D$ be a central simple division algebra of degree $d$ over its center $C$, $\sigma$ a ring endomorphism of
$D$ and $\delta:D\rightarrow D$ a \emph{left $\sigma$-derivation},
i.e. an additive map such that
$\delta(ab)=\sigma(a)\delta(b)+\delta(a)b$
for all $a,b\in D$.
The \emph{skew polynomial ring} $D[t;\sigma,\delta]$ is the
set of skew polynomials $g(t)=a_0+a_1t+\dots +a_nt^n$ with $a_i\in
D$, with term-wise addition and  multiplication  defined
via $ta=\sigma(a)t+\delta(a)$ for all $a\in D$ \cite{O1}.
 Define ${\rm Fix}(\sigma)=\{a\in D\,|\,
\sigma(a)=a\}$ and ${\rm Const}(\delta)=\{a\in D\,|\, \delta(a)=0\}$.
If $\delta=0$, define $D[t;\sigma]=D[t;\sigma,0]$.
If $\sigma=id$, define $D[t;\delta]=D[t;id,\delta]$.

 For $f(t)=a_0+a_1t+\dots +a_nt^n\in R=D[t;\sigma,\delta]$ with $a_n\not=0$, we define the {\emph degree} of $f$ as $ {\rm deg}(f)=n$ and ${\rm deg}(0)=-\infty$.
 A skew polynomial $f\in R$ is \emph{irreducible} if it is not a unit and  it has no proper factors, i.e if there do not exist $g,h\in R$ with
 $1\leq {\rm deg}(g), {\rm deg }(h)< {\rm deg}(f)$ such
 that $f=gh$ \cite[p.~2 ff.]{J96}.  We call $f\in R$ \emph{right-invariant} if $Rf$ is a left and a right ideal in $R$, and a \emph{two-sided maximal element}, if $f$ is right-invariant and $Rf$ is a non-zero maximal ideal in $R$ (equivalently, if $f\not=0$ and
 $R/Rf$ is a simple ring) \cite[p.~13]{J96}.
 Two nonzero skew polynomials $f_1,f_2 \in R$ are \emph{similar}, written $f_1\sim f_2$, if $R/Rf_1 \cong R/Rf_2$ \cite[p.~11]{J96}.

 A skew polynomial  $f \in R$ is  \emph{bounded} if there exists a nonzero polynomial $f^*  \in R$ such that $Rf^* $ is the largest two-sided ideal of $R$ contained in $Rf$. The polynomial $f^* $ is uniquely determined by $f$ up to scalar multiplication by elements of $D^\times$ and is called a \emph{bound} of $f$.

If $f\in R$ has degree $m$, then for all $g\in R$ of degree $l\geq m$,  there exist  uniquely
determined $r,q\in R$ with ${\rm deg}(r)<{\rm deg}(f)$, such that $g=qf+r.$
Let ${\rm mod}_r f$ denote the remainder of right division by $f$.
The skew polynomials $R_m=\{g\in R\,|\, {\rm deg}(g)<m\}$ of degree less that $m$ canonically represent the elements of the left
$R$-modules $R/Rf$. Furthermore, $R_m$ together with the multiplication $g\circ h=gh \, {\rm mod}_r f$
is a unital nonassociative algebra $S_f=(R_m,\circ)$  over $F_0=\{a\in D\,|\, ah=ha \text{ for all } h\in
S_f\}={\rm Comm}(S_f)\cap D$, called a \emph{Petit algebra}. When the context is clear, we simply use juxtaposition for multiplication in $S_f$.
 Note that
$ C(D)\cap{\rm Fix}(\sigma)\cap {\rm Const}(\delta)\subset F_0.$
 For all $a\in D^\times$ we have $S_f = S_{af}$, thus without loss of generality we can assume $f$ is monic when working with Petit algebras $S_f$.
  If $f$ has degree 1 then $S_f\cong D$.

\begin{lemma}\label{right divisors are left divisors} Let $R$ be a ring with no zero divisors. For all $g\in C(R)$, every right divisor of $g$ in $R$ also divides $g$ on the left.
\end{lemma}

\begin{proof}
Suppose $\gamma$ is a right divisor of $g$. Then $g=\delta\gamma$ for some $\delta\in R$. As $g$ lies in the centre of $R$, we have $\delta g = g\delta =\delta\gamma\delta$. This rearranges to $0=\delta g-\delta\gamma\delta=\delta(g-\gamma\delta).$ As $R$ contains no zero divisors and $\delta\neq 0$, it follows that $g=\gamma\delta$.
\end{proof}

\subsection{The minimal central left multiple of $f\in D[t;\sigma]$} \label{subsec:mclm}

 From now on let $\sigma$ be an automorphism of $D$ of finite order $n$ modulo inner automorphisms, i.e. $\sigma^n=i_u$ for some inner automorphism $i_u(z)=uzu^{-1}$.  Then the order of $\sigma|_C$ is $n$. W.l.o.g., we choose $u\in {\rm Fix}(\sigma)$.
Let $R=D[t;\sigma]$  and define $F=C\cap {\rm Fix}(\sigma)$.  $R$ has center
$$C(R) = F[u^{-1}t^n]=\{\sum_{i=0}^{k}a_i(u^{-1}t^n)^i\,|\, a_i\in F \}\cong F[x]$$
with $x=u^{-1}t^n$ \cite[Theorem 1.1.22]{J96}.   All polynomials $f\in R$ are bounded.

For any $f \in R=D[t;\sigma]$ with a bound in $C(R)$, we define the \emph{minimal central left multiple $mclm(f)$ of $f$ in $R$} to be the unique  polynomial of minimal degree $h \in C(R)=F[u^{-1}t^n]$ such that $h = gf$ for some $g \in R$, and such that $h(t)=\hat{h}(u^{-1}t^n)$ for some monic $\hat{h}(x) \in F[x]$. Define $E_{\hat{h}}=F[x]/(\hat{h}(x))$. If $f$ has nonzero constant term,
then $f^*\in C(R)$ \cite[Lemma 2.11]{GLN18}).
  From now on we assume that $f$ has nonzero constant term
  and denote by $h \in C(R)$, $h(t)=\hat{h}(u^{-1}t^n)$, the minimal central left multiple  of $f$.
  Then $h$ equals the bound of $f$ up to a scalar multiple from $D$.  If $f$ is irreducible in $R$, then $\hat{h}(x)$ is irreducible in $F[x]$.
 If $\hat{h}\in F[x]$ is irreducible, then $f=f_1\cdots f_r$ for irreducible $f_i\in R$ such that $f_i\sim f_j$ for all $i,j$
 (\cite{AO}, cf. \cite{PT}).

\begin{lemma}\label{Similar implies same mclm}
Let $f\in R$.
\\ (i) If $f\in R$ is irreducible, then every $g\in R$ similar to $f$ has $h$ as its minimal central left multiple.
\\ (ii) Suppose that $\hat{h}\in F[x]$ is irreducible. Then $f=f_1\cdots f_r$ for irreducible $f_i\in R$ such that $f_i\sim f_j$ for all $i,j$.
\end{lemma}

This follows easily from \cite[p.~9, Corollary 2]{Carcanague} and \cite[Theorem 1.2.9]{J96}.

 The quotient algebra $ R/Rh$ has center
$ C(R/Rh)\cong  F[x]/ (\hat{h}(x) )$, cf. \cite[Lemma 4.2]{GLN18}.
Define $E_{\hat{h}}=F[x]/ (\hat{h}(x) )$.
Suppose that  $\hat{h}(x) \neq x$, and  that
 $\hat{h}$ is irreducible in $F[x]$. Then $h$ generates a maximal two-sided ideal $Rh$ in $R$
\cite[p.~16]{J96} 
and $R/Rh$ is simple over its centre $ E_{\hat{h}}$.

\begin{theorem}\label{thm:main2} \cite{AO}
Let $f \in R=D[t;\sigma]$ be monic and irreducible of degree $m>1$ with
 minimal central left multiple $h(t)=\hat{h}(u^{-1}t^n)$. Then
${\rm Nuc}_r(S_f)$ is a central division algebra over $E_{\hat{h}}$ of degree $s=dn/k$, where $k$ is the number of irreducible factors of $h$ in $R$, and
 $$ R/Rh \cong M_k({\rm Nuc}_r(S_f)).$$
 In particular, this means ${\rm deg}(\hat{h})=\frac{dm}{s}$, ${\rm deg}(h)=km=\frac{dn m}{s}$, and
 $$[{\rm Nuc}_r(S_f) :F]= s^2\cdot \frac{dm}{s}=dms.$$
 Moreover, $s$ divides $gcd(dm,dn)$. If $f$ is not right invariant, then $k>1$ and $s\not=dn$.
\end{theorem}

We know that $[S_f:F]=[S_f:C][C:F]=d^2m\cdot n$.
Since ${\rm Nuc}_r(S_f)$ is a subalgebra of $S_f$, comparing dimensions we obtain that
$$d^2mn=[S_f:F]=[S_f:{\rm Nuc}_r(S_f)]\cdot [{\rm Nuc}_r(S_f):F]=k\cdot dms,$$
that is $[S_f:{\rm Nuc}_r(S_f)]=k.$

If $f$ is not right-invariant which is equivalent to $S_f$ being not associative, which in turn is equivalent to
$k>1$, then $s\not=dn$ looking at the degree of $h$.
Note that ${\rm deg}(h)=dn m$  is the largest possible degree of $h$.

All of the above applies in particular to the special case that $D$ is a finite field extension $K$ of $C$ of degree $n$, and $\sigma\in {\rm Aut}(K)$ has order $n$. Then $R=K[t;\sigma]$  has center
$C(R) = F[t^n]=\{\sum_{i=0}^{k}a_i(t^n)^i\,|\, a_i\in F \}= F[x]$
where $F= {\rm Fix}(\sigma)$ \cite[Theorem 1.1.22]{J96}.

%
%

\section{Constructing sets of matrices employing irreducible $f\in D[t;\sigma]$}\label{sec:matrices}

Let $R=D[t;\sigma]$ be as in Section \ref{subsec:pre} and
 $f\in R$ be an irreducible monic polynomial of degree $m>1$ with nonzero constant term and  minimal central left multiple $h(t)=\hat{h}(u^{-1}t^n)$.
 Let
$$E_f=\lbrace z(t)+Rf\,|\,z(t)=\hat{z}(u^{-1}t^n)\in F[u^{-1}t^n]\rbrace\subset R/Rf.$$
 Together with  the multiplication
  $(x+Rf)\circ(y+Rf)= (xy)+Rf$
 for all $x,y\in F[u^{-1}t^n]$, $E_f$ becomes  an $F$-algebra.

\begin{lemma} \label{z in Rf iff z in RFII} \label{{Ef isomorphic to EF}}
(i) For each $z(t)=\hat{z}(u^{-1}t^n)\in F[u^{-1}t^n]$ with $\hat{z}\in F[x]$, we have $z\in Rf$ if and only if $z\in Rh$.
 \\ (ii) $(E_f,\circ)$ is a field  isomorphic to $E_{\hat{h}}$.
\end{lemma}

\begin{proof}
(i)  As $h=gf$ for some $g\in R$, each $z\in Rh$ also lies in $Rf$.

Conversely, let $z(t)=\hat{z}(u^{-1}t^n)\in F[u^{-1}t^n]$ with $\hat{z}\in F[x]$ be such that $z\in Rf$. By the Euclidean division algorithm in $F[x]$, there exist unique
 $\hat{q}(x),\hat{r}(x)\in F[x]$ such that
 $\hat{z}=\hat{q}\hat{h}+\hat{r},$
 where ${\rm deg}(\hat{r})< {\rm deg}(\hat{h})=s$ or $\hat{r}=0$. If $\hat{r}\neq 0$, then $\hat{r}=\hat{z}-\hat{q}\hat{h}$, i.e. we found $q(t)=\hat{q}(u^{-1}t^n),r(t)=\hat{r}(u^{-1}t^n)\in F[u^{-1}t^n]$, such that
 $r(t)=z(t)-q(t)h(t)\in Rf.$
 Let $\hat{r'}(x)=r_0^{-1}\hat{r}(x)\in F[x]$, where $r_0\in F^{\times}$ is the leading coefficient of $\hat{r}(x)$,
 then $r'(t)=\hat{r'}(u^{-1}t^n) $ is monic by definition.

  As $r'(t)=\hat{r'}(u^{-1}t^n)\in Rf$, too, there exists $a(t)\in R$ such that $r'(t)=a(t)f(t)$. Thus, $r'(t)\in F[u^{-1}t^n]$ is a monic polynomial of degree less than $s$ which is right divisible by $f$. This contradicts the definition of $h$ as the minimal central left multiple of $f$. Thus we conclude that $r=0$ and $z=qh \in Rh$, as required.
   \\ (ii)  $E_f$ is a commutative associative ring with identity $1+Rf$.
Define the map $G: E_f\to E_{\hat{h}}$, $G(z+Rf)=z+Rh$
 for all $z\in F[u^{-1}t^n]$.
 $G$ is well-defined and surjective.
  For all $x,y\in F[u^{-1}t^n]$ we have
$G(x+Rf)+G(y+Rf)=(x+Rh)+(y+Rh)=(x+y)+Rh=G(x+y+Rf),$ $G(1+Rf)=1+Rh$, and
$G(x+Rf)G(y+Rf)=(x+Rh)(y+Rh)=xy+Rh=G(xy+Rf),$
yielding that $G$ is an isomorphism. To check injectivity, we note that $G(x+Rf)=0+Rh$ if and only if $x\in Rh$. By Lemma \ref{z in Rf iff z in RFII} (i), this implies $x\in Rf$ and so $x+Rf=0+Rf$.
\end{proof}

 Let $B={\rm Nuc}_r(S_f)$ and $k$ be the number of irreducible factors of $h(t)$ in $R$.

\begin{lemma} \label{prop: Vf is a B-module2}
 The left $R$-module $R/Rf$ is a right $B$-module of rank $k$ via the scalar multiplication $R/Rf\times B\longrightarrow R/Rf$, $(a+Rf)(z+Rf)=az+Rf$
for all $z\in F[u^{-1}t^n]$ and $a\in R$. We can identify $R/Rf$ with $B^{k}$ via a canonical basis.
\end{lemma}

\begin{proof}
 Since the Petit algebra $S_f=R/Rf$ with its multiplication  $ab=ab \, {\rm mod}_r\, f$ is a nonassociative unital algebra with right nucleus $B$, $R/Rf$ is a right $B$-module via the given scalar multiplication.
 As $R/Rf$ is a vector space of dimension $d^2mn$ over $F$,  $R/Rf$ is free of rank $k$ over $B$.
\end{proof}

Let $\nu\in D^\times$ and $\rho\in {\rm Aut}(D)$, and define $F'={\rm Fix}(\rho)\cap F$.
We assume in the following that $F/F'$ is finite-dimensional. Let
 $s$ be the degree of $B$ over $E_{\hat{h}}$. We assume $f$ is not right-invariant, i.e. $k>1$.

Let $l<k=dn/s$ be a positive integer. Define the set
$S_{n,m,l}(\nu,\rho, f)=\lbrace a+Rh\,|\, a\in P\rbrace\subset R/Rh,$
where
$$P=\lbrace d_0+d_1t+\dots+d_{lm-1}t^{lm-1}+\nu\rho(d_0)t^{lm}\,|\, d_i\in D\rbrace\subset D[t;\sigma].$$
$S_{n,m,l}(\nu,\rho, f)$ is a vector space over $F'$ of dimension $d^2nml[F:F']$.
 $R/Rf$ is a right $B$-module of rank $k$, as shown above.
Let $L_a:R/Rf\to R/Rf$ be the left multiplication map  $L_a(b+Rf)=ab+Rf$. Then
$L_a$ is $B$-linear, as we have
 $a( x \alpha)= ( ax) \alpha$ for all $\alpha\in B$, $a,x\in R/Rf$, and therefore
 $L_a(  x \alpha )=  L_a(x) \alpha $ for all $\alpha\in B$. Thus $L_a\in {\rm End}_{B}(R/Rf)$ and
 $$ R/Rh \cong M_k(B)\cong {\rm End}_{B}(B^k)={\rm End}_{B}(R/Rf)$$
 by Theorem \ref{thm:main2}.  Hence we have well-defined maps
$$L: S_{n,m,l}(\nu,\rho, f) \to {\rm End}_B(R/Rf), a\mapsto L_a,$$
$$\lambda: S_{n,m,l}(\nu,\rho, f) \to M_k(B), a\mapsto L_a \mapsto M_a,$$
 where $M_a$ is the matrix representing $L_a$
with respect to a $B$-basis of $R/Rf$. We denote  the image of $S_{n,m,l}(\nu,\rho,h)$ in $M_{k}(B)$ by
$$\mathcal{C}_{n,m,l}=\lbrace M_a\mid a\in S_{n,m,l}(\nu,\rho,f)\rbrace.$$
The code $\mathcal{C}=\mathcal{C}_{n,m,l}$  is $F'$-linear by construction, and a generalized rank metric code.
If $\mathcal{C}$ has minimum distance $d_{\mathcal{C}}$, the Singleton-like bound canonically generalizes to the bound
$${\rm dim}_{F'}(\mathcal{C})\leq k(k-d_{\mathcal{C}}+1)[B:F'],$$
 with $[B:F']=s[F:F']$.
If $d_{\mathcal{C}}=k-l+1$, then
${\rm dim}_{F'}(S_{n,m,l}(\nu,\rho, f))=d^2nml/dms [B:F']=d^2mnl [F:F'] =lk [B:F']=lkdms[F:F']$.
 Thus if $d_{\mathcal{C}}=k-l+1$, then $\mathcal{C}$ attains this bound and $\mathcal{C}$ is
a  \textit{maximum rank distance code} in $M_k(B)$.

We will usually deal with the case that ${\rm deg }(h)=dmn$, so that $B=E_{\hat{h}}$ is a field, $s=1,$ and
$\mathcal{C}_{n,m,l}\subset M_{dn}(E_f)$. Note that if
$l=1$ and $d_\mathcal{{C}}=k$
this generalized Singleton-like bound is achieved trivially: we obtain examples of MRD codes in $M_k(B)$.
This arises when we look at division algebras $S_{n,m,1}(\nu,\rho, f)$
and the matrices representing their left multiplication, cf. Remark \ref{remark} and Corollary \ref{cor:18}.

%
%

\section{The rank of the matrix that corresponds to the element $a+Rh$}

Let $ R=D[t;\sigma]$ be as in Section \ref{sec:matrices}, and $f\in R$ be an irreducible monic polynomial of degree $m>1$ 
with minimal central left multiple $h$. Let $B={\rm Nuc}_r(S_f)$.
We have ${\rm deg}(\hat{h})=km$ and $R/Rh\cong M_{k}(B)$ as $E_{\hat{h}}$-algebras by Theorem \ref{thm:main2}. Let $\Psi:R/Rh\to M_k(B),$ $\Psi(a+Rh)=M_a$, be this isomorphism. For $M_a\in M_k(B)$, consider the right $B$-linear map $L_{M_a}:M_k(B)\to M_k(B)$, $L_{M_a}:X\mapsto M_aX.$
 Then we obtain the following generalization of \cite{Sheekey}, Proposition 7 (which was only proved for $f$ with coefficients in a finite field, i.e. for the special case that ${\rm deg}(h)=nm$ is maximal):

\begin{theorem}\label{rank of a polynomial general}
Let ${\rm deg}(h)=km$.
 Then $M_a\in M_{k}(B)$ and
$${\rm dim}_B({\rm im}(L_{M_a}))=
 k^2 - \frac{k}{m}{\rm deg}({\rm gcrd}(a,h)),\quad  {\rm colrank}(M_a)=k-\frac{1}{m}{\rm deg}({\rm gcrd}(a,h))$$
 for all $a+Rh\in R/Rh$.
 In particular, if ${\rm deg}(h)=dmn$, then $M_a\in M_n(E_{\hat{h}})$, and
$${\rm rank}(M_a)= dn - \frac{1}{m}{\rm deg}({\rm gcrd}(a(t),h(t)).$$
\end{theorem}

\begin{proof}
For each $M_a\in M_k(B)$, define ${\rm Ann}_r(M_a)=\lbrace N\in M_k(B)\,|\, M_aN=0\rbrace$. Then ${\rm Ann}_r(M_a)$ is the kernel of the endomorphism $L_{M_a}:M_k(B)\to M_k(B)$.
By the Rank-Nullity Theorem for free right $B$-modules of finite rank \cite[ch. IV, Cor. 2.14]{H}, it follows that
$$k^2={\rm dim}_B({\rm im}(L_{M_a}))+{\rm dim}_{B}({\rm Ann}_r(M_a)).$$
We conclude that
${\rm dim}_{B}(im(L_{M_a}))= k^2-{\rm dim}_{B}({\rm Ann}_r(M_a)).$
 Now for each $b+Rh$, $M_aM_b=0$ if and only if $\Psi(a+Rh)\Psi(b+Rh)=0$. As $\Psi$ is multiplicative, this is true if and only if $\Psi((a+Rh)(b+Rh))=0$. This means $(a+Rh)(b+Rh)=0$.
Hence it is clear that ${\rm Ann}_r(M_a)\cong {\rm Ann}_r(a)$, where
$${\rm Ann}_r(a)=\lbrace b+Rh\in R/Rh \,|\, (a+Rh)(b+Rh)=0+Rh\rbrace,$$
    so ${\rm dim}({\rm Ann}_r(M_a))={\rm dim}({\rm Ann}_r(a))$.
\\
Let $\gamma=gcrd(a,h)$ so $h=\delta\gamma$ for some $\delta\in R$. As $h\in C(R)$ and $R$ is a domain, we also have $h=\gamma\delta$ by Lemma \ref{right divisors are left divisors}. Let $b\in R$ be the unique element such that $a=b\gamma$. Then $gcrd(b,\delta)=1$, else $\gamma$ is not the greatest common right divisor of $a$ and $h$.\\
Let $v\in R$. By the left Euclidean division algorithm, there exist unique $u,w\in R$ such that $v=\delta u+w$ where $deg(w)<{\rm deg}(\delta)$ and ${\rm gcld}(w,\delta)=1$. It follows that
$av = a\delta u+aw = b\gamma\delta u+ b\gamma w = bhu+b\gamma w,$
therefore $av+Rh=b\gamma w+Rh$.\\
Suppose $b\gamma w\equiv 0\: {\rm mod}_rh$. As $gcrd(b,\delta)=1$, there exist $c,d\in R$ such that $cb+d\delta=1$, so $cb\gamma+d\delta\gamma=\gamma.$
As $\delta\gamma=h$, this implies $cb\gamma\equiv\gamma\: {\rm mod}_rh$. Hence
$\gamma w\equiv cb\gamma w\equiv 0\: {\rm mod}_rh.$
However, ${\rm deg}(w)<deg(\delta)$ so ${\rm deg}(\gamma w)< {\rm deg}(\gamma\delta)={\rm deg}(h)$; due to this, $\gamma w\equiv 0\: {\rm mod}_rh$ implies that $\gamma w=0$. As $\gamma\neq 0$ and $R$ is a domain, we conclude that $w=0$.\\
Hence, $(a+Rh)(v+Rh)=0+Rh$ if and only if $v=\delta u$ where ${\rm deg}(u)<{\rm deg}(\gamma)$. As $\delta$ is uniquely defined by $a$ and $h$, every element of ${\rm Ann}_r(a)$ is determined by $u\in R$ such that ${\rm deg}(u)<{\rm deg}(\gamma)$.
Thus
  \begin{align*}
  {\rm Ann}_r(a)=&\lbrace v+Rh\in R/Rh\,|\, (a+Rh)(v+Rh)=0+Rh\rbrace\\
  =&\{\delta u\,|\,u\in R, {\rm deg}(u)<deg(\gamma) \}
  \cong R_{deg(\gamma)}=\{g\in R\,|\, {\rm deg}(g)< deg(\gamma)\}.
  \end{align*}
As $\lbrace 1, t,\dots, t^{ deg(\gamma)-1}\rbrace$ is a $D$-basis
 for the free left $D$-module  $R_{\gamma}$, it follows that ${\rm dim}_{D}({\rm Ann}_r(a))={\rm deg}(\gamma)$,
 so ${\rm dim}_{F}({\rm Ann}_r(a))={\rm deg}(\gamma)d^2n$.
  Since ${\rm dim}_{E_{\hat{h}}}(B)=s^2=d^2n^2/k^2$ and $[E_{\hat{h}}:F]=km/n$, we obtain 
 ${\rm dim}_F(B)=d^2mn/k$.
Hence we get
$${\rm dim}_{B}({\rm Ann}_r(a))=\frac{{\rm deg}(\gamma)d^2nk}{d^2mn}=\frac{{\rm deg}(\gamma)k}{m},$$
and so
$${\rm dim}_{B}({\rm im}(L_A))= k^2-{\rm dim}_{B}({\rm Ann}_r(M_a))=k^2-\frac{k}{m}{\rm deg}(\gamma).$$
Let $\underline{c_i}$, respectively $\underline{r_i}$, denote the columns and rows of $M_a$  and $\underline{x_i}$ denote the columns of $X$. Computing the matrix using dot product notation we have
\begin{equation*}
M_aX=\left(
\begin{array}{ccc}
\underline{r_1}\cdot\underline{x_1}&\dots&\underline{r_1}\cdot\underline{x_k}\\
\vdots & \ddots & \vdots\\
\underline{r_k}\cdot\underline{x_1}&\dots& \underline{r_k}\cdot\underline{x_k}
\end{array}\right)
\end{equation*}
The $i^{th}$ column of $M_aX$ is equal to \begin{equation*}
\begin{pmatrix}
\underline{r_1}\cdot\underline{x_i}\\
\vdots\\
\underline{r_k}\cdot\underline{x_i}
\end{pmatrix}
=\underline{c_1}\lambda_1+\dots+\underline{c_k}\lambda_k
\end{equation*}
for some $\lambda_j\in B$. Hence the dimension of the right $B$-module generated by the $i^{th}$ column of
$M_aX$ is exactly the column rank of $M_a$. As there are $k$ columns of $M_aX$, it follows that ${\rm dim}_B({\rm im}(L_{M_a}))=k \,{\rm colrank}(M_a)$.
\end{proof}

All of the above applies in particular to the special case that
 $K/F$ is a field extension and $\sigma\in {\rm Aut}_F(K)$ of finite order $n$, $R=K[t;\sigma]$ and  $C(R) = F[t^n]\cong F[x]$. Let $f\in R$ be a monic irreducible polynomial of degree $m>1$, $B={\rm Nuc}_r(S_f)$,
and $h(t)=\hat{h}(t^n)$ its minimal central left multiple, ${\rm deg}(\hat{h})=km$.
Then $\Psi:R/Rh\to M_k(B),$ $\Psi(a+Rh)=M_a$ is an $E_f$-algebra isomorphism. For each $M_a\in M_k(B)$, we have the endomorphism $L_{M_a}:M_k(B)\to M_k(B)$ by $L_{M_a}:X\mapsto M_aX.$
 Analogously to Theorem \ref{rank of a polynomial general} we can prove:

\begin{theorem} [for finite fields and thus ${\rm deg}(h)=nm$ maximal, cf. \cite{Sheekey}, Proposition 7] \label{rank of a polynomial km}
Suppose that ${\rm deg}(h)=km$, then for all $a+Rh\in R/Rh$ we have
$${\rm dim}_B({\rm im}(L_{M_a}))=
 k^2 - \frac{k}{m}{\rm deg}({\rm gcrd}(a,h)) ,\quad  {\rm colrank}(M_a)=k-\frac{1}{m}{\rm deg}({\rm gcrd}(a,h)).$$
 In particular, if ${\rm deg}(h)=mn$ then $M_a\in M_n(E_{\hat{h}})$ and
${\rm rank}(M_a)= n - \frac{1}{m}{\rm deg}({\rm gcrd}(a,h))$.
 \end{theorem}
This generalizes  \cite[Remark 6]{Sheekey}.

%
%

\section{Using the norm of $D(t; \sigma )$ to investigate $f$}

\subsection{The algebra $(D(x), \widetilde{\sigma}, ux )$}\label{sec_norm1}

Let $C/F$ be a finite cyclic field extension of degree $n$ with ${\rm Gal}(C/F) = \langle \sigma\rangle$. Let $D$ be a finite dimensional division algebra of degree $d$ with center $C$ and suppose that $\sigma$ extends to a $C$-algebra automorphism of $D$ that we call $\sigma$, too.
Let $R=D[t;\sigma]$ as in Section \ref{sec:matrices}.
Then there exists $u \in D^\times$ such that $\sigma^n = i_u$ and $\sigma(u) = u$. These two relations determine $u$ up to multiplication with elements from $F^\times$ \cite[Lemma 19.7]{Pierce}.

 The quotient algebra $(D,\sigma,a)=D[t;\sigma]/(t^n-a)D[t;\sigma]$, where $f(t)=t^n-a\in D[t;\sigma]$ with
$d\in F^\times$, is called a \emph{generalized cyclic algebra}.
The special case where $D=C$  yields the cyclic algebra $(C/F,\gamma,a)$ \cite[p.~19]{J96}.

Let
 $D(t;\sigma) = \{f/g \,|\, f \in D[t;\sigma], g \in C(D[t;\sigma])\}$
be the ring of central quotients of $D[t;\sigma]$. Let $\widetilde{\sigma}$ denote the extension of $\sigma$ to $D(x)$ that fixes $x$ \cite[Lemma 2.1.]{TH}.
Then $C(D(t; \sigma)) = {\rm Quot}(C(D[t;\sigma])) = F(x)$,  $x =u^{-1}t^n$, is the center of $D(t;\sigma)$,
 where ${\rm Quot}(U)$ denotes the quotient field of an integral domain $U$. More precisely,
 $D(t; \sigma ) \cong (D(x), \widetilde{\sigma}, ux )$
 is a generalized cyclic algebra of degree $dn$ over its center $F(x)$ and a division algebra \cite[Theorems 2.2, 2.3]{TH}.

   Let $N$ be the reduced norm of $(D(x), \widetilde{\sigma}, ux )$.

\begin{lemma}\label{le:easy2}
Let $f\in R$. If $N(f)$ is irreducible in $F[x],$ then $f$ is irreducible in $R$.
\end{lemma}

\begin{proof}
  If $f=gp$ for $g,p\in R$ then $N(f)=N(g)N(p)$ is reducible in $F[x],$ since both $N(g)$ and $N(p)$ lie in $F[x]$, which immediately yields the assertion.
\end{proof}

From now on, we assume  that
$$D=(E/C,\gamma,a) \text{ is a cyclic  division algebra over } C \text{ of degree } d,$$
$$\sigma|_E\in {\rm Aut}(E) \text{ such that }\gamma\circ\sigma=\sigma\circ\gamma \text{ and } u\in E.$$
 Then $\sigma|_E$ has order $n$. Write $m=kn+r$ for some $0\leq r<n$. Let $f=\sum_{i=0}^ma_it^i\in R$ be a polynomial such that $a_0\not=0$
 and $h\in R$ be the minimal central left multiple of $f$ in $R$.

\begin{theorem} \label{thm:norm3}\cite{PT}
  For $f\in E[t;\sigma]\subset D[t;\sigma]$, we have
  $$N(f(t))=N_{E/F}(a_0)+\dots +(-1)^{dr(n-1)}N_{E/F}(a_m)N_{E/C}(u)^{m}x^{dm}.$$
  \end{theorem}

  \begin{theorem}  \label{Sheekey Theorem 3 generalisation}
 Suppose that  ${\rm deg}(h)=dmn$.
 \\ (i)\cite[Theorem 14 (i)]{PT} If $\hat{h}$ is irreducible in $F[x]$ then $f$ is irreducible in $R$.
 \\ (ii)\cite[Theorem 14 (ii)]{PT} If $f$ is irreducible then $N(f)$ is irreducible in $F[x]$.
 \\ (iii) If  $f\in E[t;\sigma]$, then  $N(f)= (-1)^{dr(n-1)}N_{E/F}(a_m)N_{E/C}(u)^{m} \hat{h}$ and
 $$N_{E/F}(a_0)=(-1)^{dr(n-1)}N_{E/F}(a_m)N_{E/C}(u)^{m}h_0,$$
 if $h_0$ denotes the constant term of $\hat{h}$.
 \end{theorem}

  \begin{proof}
 (iii) By Theorem \ref{thm:norm3} we have ${\rm deg}(N(f))=dmn$ in $R$. $N(f)$ is a two-sided multiple of $f$ in $R$, therefore the bound $f^*$ of $f$  divides $N(f)$ in $R$. Since  $(f,t)_r=1,$ $f^*\in C(R)$ and therefore $f^*$ equals $h$ up to some  factor in $F^\times$. Thus
  $h(t)=\hat{h}(u^{-1}t^n)$ must divide $N(f)$ in $R$. Write $N(f)=g (t)h(t)$ for some $g\in R$.
Comparing degrees in $R$ we obtain ${\rm deg} N(f)={\rm deg}(g(t))+dmn=dmn$, which implies ${\rm deg}(g)=0$, i.e. $g(t)=a\in A^\times$.
  This implies that $N(f)=ah(t)=a \hat{h}(u^{-1}t^n)$. Comparing highest coefficients of $N(f)$ and $a\hat{h}$ yields that
   $a=(-1)^{dr(n-1)}N_{E/F}(a_m)N_{E/C}(u)^{m}$ by Theorem \ref{thm:norm3}, so that comparing constant terms we get that $N_{E/F}(a_0)=(-1)^{dr(n-1)}$ $N_{E/F}(a_m)N_{E/C}(u)^{m}h_0$, if $h_0$ is the constant term of $\hat{h}(x)$.
  \end{proof}

\begin{theorem}\label{divisors in h when R=D[t,sigma]}
Let $f\in E[t;\sigma]\subset R$ be monic and irreducible of degree $m$. 
 Let  ${\rm deg}(\hat{h})=dm$ and suppose that all the monic polynomials similar to $f$ lie in $E[t;\sigma]$.
 If $g$ is a monic divisor of $h$ in $R$ of degree $lm$, then
 $$N_{E/F}(g_0)=N_{E/F}(a_0)^l.$$
\end{theorem}

\begin{proof}
We know that  $h(t)=\hat{h}(u^{-1}t^n)$, with $\hat{h}(x)$ irreducible in $F[x]$, since $f$ is irreducible. Thus $h$ is a t.s.m. element in Jacobson's terminology \cite{J96} and the irreducible factors $f_1(t),\dots, f_k(t)$ of any decomposition of $h(t)$ are all similar, and are all similar to $f$, as $f$ must be one of them by the definition of $h$.
Now $g(t)$ is a monic divisor of $h$.  Thus we can decompose $g(t)$ into a product of irreducible factors and up to similarity the irreducible factors of $g$ will be the same as suitably chosen irreducible factors of $h$ by \cite[Theorem 1.2.9.]{J96}. Hence w.l.o.g. $g=f_1 f_2\cdots f_l$, where the $f_i$ are irreducible in $R$ and $f_i$ is similar to $f$ for all $i=1,2,\dots,l$ \cite[Theorem 1.2.19]{J96}. Thus by Lemma \ref{Similar implies same mclm}, the minimal central left multiple of each $f_i$ is equal to $h$.
Since $f$ is monic, we may assume w.l.o.g. that all $f_i$ are monic.
By Theorem \ref{Sheekey Theorem 3 generalisation} and since all $f_i\in E[t;\sigma]$ by our assumption, this implies that $N_{E/F}(f_i(0))=(-1)^{dm(n-1)}N_{E/C}(u)^{m}h_0=N_{E/F}(a_0)$. As the constant term of $g$ is equal to $\prod_{i=1}^l f_i(0)$, we see that
$$N_{E/F}(g_0)=\prod_{i=1}^l N_{E/F}(f_i(0))=[ (-1)^{dm(n-1)}N_{E/C}(u)^{m}h_0]^l$$
$$=(-1)^{l dm(n-1)}N_{E/C}(u)^{lm}h_0^l=N_{E/F}(a_0)^l.$$
\end{proof}

 We are not able to say if the assumptions on the $f_i$'s in the above result is empty or trivial.

\subsection{The algebra $(K(x)/F(x),\widetilde{\sigma},x)$}


Let $K/F$ be a cyclic field extension of degree $n$ with ${\rm Gal}(K/F)=\langle \sigma\rangle$,
 $R=K[t;\sigma]$ and $x=t^n$. We now look at  the cyclic algebra $(K(x)/F(x),\widetilde{\sigma},x)$ (this case corresponds to $D=C$ in the previous Section).
Let $N$ be the reduced norm of $(K(x)/F(x),\widetilde{\sigma},x)$ over $F(x)$ (cf. also \cite[Proposition 1.4.6]{J96}). We have $\widetilde{\sigma}|_K=\sigma$, and $N$ is a nondegenerate form of degree $n$.
Let $f=\sum_{i=0}^ma_it^i\in R$ be a polynomial of degree $m$ such that $a_0\not=0$
and $h\in R$ be the minimal central left multiple of $f$ in $R$. Then $N(f(t))=N_{K/F}(a_0)+\dots + (-1)^{m(n-1)}N_{K/F}(a_m)x^m$ \cite[Theorem 3]{PT}.

  \begin{theorem} \label{thm:norm4}
 Suppose that ${\rm deg}(h)=mn$.
 \\ (i) \cite[Theorem 6 (i)]{PT} If $\hat{h}$ is irreducible in $F[x]$ then $f$ is irreducible in $R$.
 \\ (ii)  \cite[Theorem 6 (ii)]{PT} If $f$ is irreducible then $N(f)$ is irreducible in $F[x]$.
 \\ (iii) $N_{K/F}(a_0)=(-1)^{m(n-1)}h_0,$ if $h_0$ denotes the constant term of $\hat{h}$.
 \end{theorem}

  Theorem \ref{thm:norm4} (iii) is proved analogously as Theorem \ref{Sheekey Theorem 3 generalisation} (iii).

\begin{theorem}\label{thm: divisors degree lm} (cf. \cite[Theorem 5]{Sheekey}\label{Sheekey Theorem 5} for finite fields, the proof is the same)
Suppose that $f$ is not right-invariant.
If ${\rm deg}(h)=mn$ and $g$ is a monic divisor of $h(t)$ in $R$ of degree $ml$, then
$$N_{K/F}(g_0)=N_{K/F}(a_0)^l.$$
\end{theorem}

%
%

\section{Division algebras and MRD codes employing $f\in R$}\label{sec:division}

\subsection{The case that $f\in D[t;\sigma]$}

Let $f\in R=D[t;\sigma]$ be a monic polynomial of degree $m$.
Let $\rho\in {\rm Aut }(D)$, $\nu\in D$ and $F'={\rm Fix}(\rho)\cap F$ where $F=C\cap {\rm Fix}(\sigma)$.
Let $b(t), c(t)\in R_m=\{g\in R\,|\, {\rm deg}(g)<m\}$ and $b_0$ be the constant term of $b(t)$. Then  the multiplication defined via
$$b(t)\circ c(t)=(b(t)+\nu\rho(b_0)t^m ) c(t)\,\,{\rm mod}_r f$$
makes $R_m$ into a non-unital nonassociative ring $(R_m,\circ)$.
When the context is clear, we will drop the $\circ$ notation and simply use juxtaposition.
 $(R_m,\circ)$ is an algebra over $F'$.

\begin{example}
 If $f(t)=t-c\in D[t;\sigma]$ for some $c\in D$, $\nu\neq 0$, then
  $(R_m,\circ)$ has the multiplication
  \begin{align*}
a\circ b =& (a+\nu\rho(a)t)b)\: \,\,{\rm mod}_r f\\
=& ab+\nu\rho(a)\sigma(b)t\: \,\,{\rm mod}_r f\\
=& ab+\nu\rho(a)\sigma(b)c
\end{align*}
for all $a,b\in D$.  This generalizes the algebras studied in \cite{P15}.
If $R=K[t;\sigma]$ for some finite field extension $K/F$, this is the multiplication of Albert's twisted semifields \cite{A}. If $F/F'$ is finite and $(R_m,\circ)$ is not a division algebra,
 $a\circ b=0$ for some non-zero $a,b\in D$, if and only if $ ab=-\nu\rho(a)\sigma(b)c$. Taking norms of both sides and cancelling $N_{D/F'}(ab)$ on both sides, we obtain that $N_{D/F'}(-\nu c)=(-1)^{d^2n[F:F']} N_{D/F'}(\nu c)=1.$
 Thus  if $F/F'$ is finite and $N_{D/F'}(\nu c)\neq (-1)^{d^2n[F:F']}$ then $(R_m,\circ)$ is a division algebra.
\end{example}

From now on for the rest of the paper, we  again assume that  $f$ is an irreducible monic polynomial of degree $m>1$,
 $(f,t)_r= 1$, and that $h$ is the minimal central left multiple of $f$. Let $F/F'$ be finite-dimensional,
and
$$P=\lbrace d_0+d_1t+\dots+d_{lm-1}t^{lm-1}+\nu\rho(d_0)t^{lm}\,|\, d_i\in D\rbrace\subset D[t;\sigma].$$

\begin{theorem}  \label{thm:division3}
Let $l=1$. Then: \\
 (i)  Let $b(t)\in R_m$ with constant coefficient $b_0$.  If $b(t)+\nu \rho (b_0)t^m\in P$ is reducible in $R$, then $b(t)$ is not a left zero divisor in $(R_m, \circ)$.
\\ (ii) If $\nu=0$ then $(R_m, \circ)$ is a division algebra over $F'$, which for $m\geq 2$ is a  Petit algebra.
\\ (iii) If $P$ does not contain any polynomial similar to $f$,
then
$(R_m, \circ)$ is a division algebra over $F'$.
\end{theorem}

Note that $f$ may be right-invariant.

\begin{proof}
Suppose that there are $b(t)=b_0+b_1t+\dots+b_{m-1}t^{m-1}, c(t)\in R_m$, such that
$$b(t)\circ c(t)=(b(t)+\nu \rho (b_0)t^m) c(t){\rm mod}_r f =0.$$
Then there exists $g\in R_m$ such that $(b(t)+\nu \rho (b_0)t^m) c(t)=g(t)f(t).$
Since $f$ is irreducible and of degree $m$, while $deg(c) < m$, $f$ must be similar to an irreducible factor of $b(t)+\nu \rho (b_0)t^m$, because of the uniqueness of an irreducible decomposition in $R$ up to similarity.
But $b(t)+\nu \rho (b_0)t^m$ has degree at most $m$, so $f$     is similar to $b(t)+\nu \rho (b_0)t^m$. Thus $b(t)+\nu \rho (b_0)t^m$ must have degree $m$ and be irreducible as well.
Hence if $b(t)+\nu \rho (b_0)t^m$ is not similar to $f$ then $b(t)+\nu \rho (b_0)t^m$ is not a left zero divisor in $(R_m, \circ)$. This happens for instance, if $\nu=0$ or if $b(t)+\nu \rho (b_0)t^m$ is reducible.
Moreover, $(R_m, \circ)$ is a division algebra if $P$ does not contain any polynomial similar to $f$.
\end{proof}

We are again not able to say if the assumptions on the $f_i$'s in the following result is empty or trivial.

\begin{theorem}  \label{thm:division2}
Let $D=(E/C,\gamma,a)$ be a cyclic division algebra over $C$ of degree $d$ such that
$\sigma|_E\in {\rm Aut}(E)$ and $\gamma\circ\sigma|_E=\sigma|_E\circ\gamma$.
Suppose that
 $\sigma^n(z)=u^{-1}zu$ with $u\in E$.

Let $f(t)=\sum_{i=0}^ma_it^i\in E[t;\sigma]\subset D[t;\sigma]$ be monic and irreducible, 
and
let ${\rm deg}(h)=dmn$. Suppose that all monic $f_i$ similar to $f$  lie in $E[t;\sigma]$. Then
$(R_m, \circ)$ is a division algebra over $F'$, if one of the following holds:
\\ (i)  $\nu\not\in E$ and $\rho|_E\in {\rm Aut}(E)$.
\\ (ii) $\nu\in E^\times$  and $\rho|_E\in {\rm Aut}(E)$,
such that
$$N_{E/F'}(a_0) N_{E/F'}(\nu)\neq 1.$$
\end{theorem}

Note that our global assumption that $\sigma^n(z)=u^{-1}zu$ for all $z\in D$, so that $\sigma^n(e)=u^{-1}eu=e$ for all $e\in E$,  forces $(\sigma|_E)^n=id$.

\begin{proof}
By Theorem \ref{thm:division3},
$(R_m, \circ)$ is a division algebra, if the set $P$ with $l=1$ does not contain any polynomial similar to $f$. All polynomials similar to $f$ are irreducible factors of $h(t)$, so $(R_m, \circ)$ is a division algebra, if $P$ does not contain any irreducible factor of $h(t)$. Suppose that $P$ contains an irreducible factor $g$ of $h$ with constant term $g_0$. Then $g$ has degree $m$ as it is similar to $f$. Let  $g_{m} t^{m}$ be its highest coefficient, so that $g_{m}^{-1}g$ is a monic divisor of $h$.

By Theorem \ref{Sheekey Theorem 3 generalisation} and since $g\in E[t;\sigma]$  by assumption, this implies
$$N_{E/F}(g_0g_m^{-1})=(-1)^{m(n-1)}h_0=N_{E/F}(a_0)$$
and in particular, that $g_0$ and $g_m$ are both non-zero. Since $g\in P$,
we also have $g_{m}=\nu\rho(g_0)$.
\\ Suppose $\nu\not\in E$ and $\rho(E)\subset E$. Since the coefficients of the $f_i$ all lie in $E$ we have
$g_{m}\not=\nu\rho(g_0)$ which yields a contradiction. Hence there is no divisor $g$ of $h$ in $P$ and $S$ is a division algebra.
\\
Suppose that $\nu\in E^\times$ and $\rho(E)\subset E$.
Substituting $g_{m}=\nu\rho(g_0)$  into the above equation yields
$$N_{E/F}(g_0)=N_{E/F}(a_0) N_{E/F}(\nu\rho(g_0)).$$
Applying $N_{F/F'}$ to both sides implies that
$$N_{E/F'}(g_0)=N_{F/F'}(N_{E/F}(a_0)) N_{E/F'}(\nu\rho(g_0)).$$
Now $N_{E/F'}(\rho(g_0))=N_{E/F'}(g_0)$, so we can cancel the non-zero term $N_{E/F'}(g_0)$ to obtain
$1=N_{E/F'}(a_0) N_{E/F'}(\nu).$
\end{proof}

\begin{remark}\label{remark}
 Let $S=S_{n,m,1}(\nu,\rho, f)=\lbrace a+Rh\,|\, a\in P\rbrace$.
 We can use $\mathcal{C}(S)\subset M_k(B)$ to define a multiplication on $B^m$. As ${\rm dim}_F(D)=d^2n$ and ${\rm dim}_F(B)=d^2mn/k$, there exists an $F$-vector space isomorphism between $D^m$ and $B^k$. Similarly, there exists an isomorphism  $G:V_f\to B^k$, $G(a+Rf)=\underline{a}$.
Define $\ast:B^k\times B^k\to B^k$ by
$$\underline{a}\ast\underline{b}=M_a\cdot \underline{b}$$
for all $\underline{a},\underline{b}\in B^k$, where $M_a\in \mathcal{C}(S)$ is the representation of the map $L_{a(t)+\nu\rho(a_0)t^m}\in {\rm End}_B(R/Rf)$ induced by $G$.
(Each $a\in R_m$ corresponds to a map $L_{a(t)+\nu\rho(a_0)t^m}$. As ${\rm End}_B(R/Rf)\cong M_k(B)$ and ${\rm dim}(R_m)={\rm dim}(\mathcal{C}(S))$, there is a canonical bijection between $L_{a(t)+\nu\rho(a_0)t^m}$ and $M_a$.)
As $M_a$ represents $L_a\in\text{End}_B(R/Rf)$, $(B^k,\ast)$ is  isomorphic to $R/Rf$ equipped with the multiplication $(a+Rf)(b+Rf)=L_{a(t)+\nu\rho(a_0)t^m}(b+Rf)$. Thus $(R_m,\circ)$ and $(B^k,\ast)$ are
 isomorphic algebras and
 $S_{n,m,1}(\nu,\rho, f)$ is the same algebra as $(R_m, \circ)$.
\end{remark}

If $l=1$ we write $S(\nu,\rho, f)=S_{n,m,1}(\nu,\rho, f)$ for $(R_m,\circ)$.
$S(\nu,\rho, f)$ is a division algebra if and only if every matrix in $\mathcal{C}_{n,m,1}$ has full column rank.
It then canonically defines an $F'$-linear $MRD$ code in $M_k(B)$, $B={\rm Nuc}_r(S_f)$. Therefore we obtain from all of the above results:

\begin{corollary}\label{cor:18}
Let $D=(E/C,\gamma,a)$ be a cyclic division algebra over $C$ of degree $d$ such that
$\sigma|_E\in {\rm Aut}(E)$ and $\gamma\circ\sigma=\sigma\circ\gamma$.
Suppose that
 $\sigma^n(z)=u^{-1}zu$ with $u\in E$.

Let $f=\sum_{i=0}^ma_it^i\in  R$ be  monic and irreducible of degree $m$. Then $B$ is a division algebra over $E_{\hat{h}}$ and $S(\nu,\rho, f)$
defines an $F'$-linear  MRD-code in $M_k(B)$ with minimum distance $k$, if one of the following holds:
 \\ (i)  $\nu=0$. Then $S(\nu,\rho, f)$ is a (unital) Petit algebra.
 \\ (ii) $P$ does not contain any polynomial similar to $f$.
 \\ (iii) Suppose $\rho|_E\in {\rm Aut}(E)$, $f=\sum_{i=0}^ma_it^i\in E[t;\sigma]\subset R$, ${\rm deg}(h)=dmn$,  
 all the monic polynomials similar to $f$ lie in $E[t;\sigma]$, and one of the following holds:
 \\ (a) $\nu\not\in E$,
  \\ (b) $N_{E/F'}(\nu)N_{E/F'}(a_0)\neq 1.$
\\
Then 
we get an $F'$-linear MRD-code in $M_{dn}(E_{\hat{h}})$ with minimum distance $dn$.
  \end{corollary}

The case $\nu=0$ produces the MRD codes which are associated with the unital Petit algebras. They can be viewed as generalized Gabidulin codes.

More generally, we can also construct MRD-codes for $l>1$.
Let $f\in R$ not be right-invariant, and let $l<k$ be a positive integer.

\begin{theorem} \label{thm:MRD}
Suppose that $P$ does not contain any polynomial of degree $lm$, whose irreducible factors are all similar to $f$. Then the set
$S_{n,m,l}(\nu,\rho, f)$
defines an $F'$-linear MRD-code
in $M_k(B)$ with minimum distance $k-l+1$. In particular, if ${\rm deg}(h)=dmn$,
then this code is an $F'$-linear  MRD-code in $M_{dn}(E_{\hat{h}})$ with minimum distance $dn-l+1$.
\end{theorem}

We are not able to say if the assumption on $P$ can be satisfied in this general setup. It is satisfied in the case considered in
\cite[Theorem 7]{Sheekey}.

\begin{proof}
We have to show that the minimum column rank of the matrix corresponding to a nonzero element in $S_{n,m,l}(\nu,\rho, f)$ is $k-l+1$. By Theorem \ref{rank of a polynomial general}, this is equivalent to finding an element $g\in A$ such that the greatest common right divisor of $g$ and $h$ has degree at most $(l-1)m$. Suppose towards a contradiction that there exists $g\in A$ such that ${\rm deg}({\rm gcrd}(g,h))=lm$; since ${\rm deg}(g)\leq lm$, it follows that $g$ must be a divisor of $h$.
As any divisor of $h$ is a product of irreducible polynomials similar to $f$,  $g$ must be a product of polynomials similar to $f$. This contradicts our assumption, so any matrix has rank at least $k-l+1$.
\end{proof}

\begin{theorem}(for  $f\in K[t;\sigma]$, $K$ a finite field, this is \cite[Theorem 7]{Sheekey}) 
Let  $f=\sum_{i=0}^ma_it^i \in E[t;\sigma]\subset R=D[t;\sigma]$ be monic irreducible, and
let ${\rm deg}(h)=dmn$. Suppose that all monic $f_i$ similar to $f$  lie in $E[t;\sigma]$.
Then $S_{n,m,l}(\nu,\rho, f)$
defines an $F'$-linear MRD code in $M_{dn}(E_{\hat{h}})$ with minimum distance  $dn-l+1$, 
if one of the following holds:
\\ (i) $\nu=0$
\\ (ii)  $\nu\not \in E$ and $\rho|_E\in {\rm Aut}(E)$.
\\ (iii) $\nu \in E$, $\rho|_E\in {\rm Aut}(E)$ and
$N_{E/F'}(\nu)N_{E/F'}(a_0)^l\neq 1.$
\end{theorem}

The proof is straightforward.

\subsection{The case $R=K[t;\sigma]$}

Let $f=\sum_{i=0}^ma_it^i\in R=K[t;\sigma]$ be an irreducible monic polynomial of degree $m$  with minimal central left multiple $h$.
Suppose throughout this section that $F/F'$ is a finite field extension, $\nu\in K$, and $\rho\in {\rm Aut}(K)$. Let $1<l<k$ and
$S_{n,m,l}(\nu,\rho, f)=\lbrace a+Rh\,|\, a\in P\rbrace\subset R/Rh,$
where
 $P=\lbrace d_0+d_1t+\dots+d_{lm-1}t^{lm-1}+\nu\rho(d_0)t^{lm}\,|\, d_i\in K \rbrace.$
  Then we obtain the following results:

\begin{theorem} \label{thm:divisionK} Let $l=1$.\\
 (i)  Let $b(t)\in R_m$ with constant coefficient $b_0$. If $b(t)+\nu \rho (b_0)t^m\in P$ is reducible in $R$, then $b(t)$ is not a left zero divisor in $S(\nu,\rho, f)$.
\\ (ii) If $\nu=0$ then $S(\nu,\rho, f)$ is a division algebra over $F'$, a unital Petit algebra.
\\ (ii) If $P$ does not contain any polynomial similar to $f$,
then $S(\nu,\rho, f)$ is a division algebra over $F'$.
\end{theorem}

The proof is analogous to the one of Theorem \ref{thm:division3}.
Note that $f$ may be right-invariant here. Using Theorems \ref{Sheekey Theorem 3 generalisation} and  \ref{Sheekey Theorem 5}, we obtain
(for finite fields, cf. \cite{Sheekey}, the proof is analogous):

\begin{theorem}\label{thm: divison algebra K}
Suppose that ${\rm deg}(h)=mn$. Then
$S(\nu,\rho, f)$ is a division algebra over $F'$ if
$$N_{K/F'}(a_0) N_{K/F'}(\nu)\not=1.$$
\end{theorem}

\begin{corollary}
 $B={\rm Nuc}_r(S_f)$ is a division algebra and the left multiplication of the algebra
$S(\nu,\rho, f)$ 
defines an $F'$-linear MRD-code in $M_k(B)$ with minimum distance $k$,  if one of the following holds:
 \\ (i) $\nu=0$.
 \\ (ii) $P$ does not contain any polynomial similar to $f$.
 \\ (iii) ${\rm deg}(h)=mn$ and $\nu \in K$  such that
$N_{K/F'}(\nu)\neq 1/N_{K/F'}(a_0).$
In this case, the algebra $S(\nu,\rho, f)$ defines an $F'$-linear MRD-code
 in $M_n(E_{\hat{h}})$ with minimum distance $n$.
\end{corollary}

Note that the condition on $f$ in (iii) is satisfied for all $f$ if $gcd(m,n)=1$ or if $n$ is prime.

We now look at the case that   $1<l<k$ and also assume that $f$ is not right-invariant.

\begin{theorem} (for finite fields, cf. \cite[Theorem 7]{Sheekey})
If ${\rm deg}(h)=mn$, then the set $S_{n,m,l}(\nu,\rho, f)$
defines an $F'$-linear MRD-code in $M_n(E_{\hat{h}})$ with minimum distance $n-l+1$ for any $\nu \in K$  such that
$$N_{K/F'}(\nu)\neq 1/N_{K/F'}(a_0)^l.$$
\end{theorem}

Note that $k=n$ here since ${\rm deg}(h)=mn$.

\begin{corollary}
 The set $S_{n,m,l}(\nu,\rho, h)$
defines an $F'$-linear MRD-code in $M_n(E_{\hat{h}})$ with minimum distance $n-l+1$, if one of the following holds:
\\ (i) ${\rm deg}(h)=mn$ and $\nu=0$,
\\ (ii) $n$ is prime or $gcd(m,n)=1$, and $1\not=N_{K/F'}(\nu)N_{K/F'}(a_0)^l$
\\ (iii) ${\rm deg}(h)=mn$ and $N_{K/F'}(\nu)\not \in ({F'}^\times)^l$.
\end{corollary}

The codes $S_{n,m,l}(0,\rho, h)$ generalize the  Gabidulin codes constructed in \cite{Gabidulin} that go back to \cite{D}.

Note that $N_{K/F'}(\nu)\not \in ({F'}^\times)^l$ implies $N_{K/F'}(\nu)\not =N_{K/F'}(a_0)^l$ for any $f$. Thus (ii) implies (iii) above.

\begin{theorem}
Suppose that $P$ does not contain any polynomial of degree $lm$, whose irreducible factors are all similar to $f$. Then the set
$S_{n,m,l}(\nu,\rho, f)$
defines an $F'$-linear MRD-code in $M_k(B)$  with minimum distance $k-l+1$.
\\ In particular, if ${\rm deg}(h)=mn$ then $S_{n,m,l}(\nu,\rho, f)$
defines an $F'$-linear MRD-code in $M_n(E_{\hat{h}})$ with minimum distance $n-l+1$.
\end{theorem}

\begin{proof}
We have to show that the minimum column rank of the matrix corresponding to a nonzero element in $S_{n,m,l}(\nu,\rho, f)$ is $k-l+1$. By Theorem \ref{rank of a polynomial km},
  this is equivalent to finding an element $g\in P$ such that the greatest common right divisor of $g$ and $h$ has degree at most $(l-1)m$. Suppose towards a contradiction that $\text{deg}(gcrd(g,h))=lm$; since $\text{deg}(g)\leq lm$, it follows that $g$ must be a divisor of $h$.

As any divisor of $h$ is a product of irreducible polynomials similar to $f$,  $g$ must be a product of polynomials similar to $f$. This contradicts our assumption, so any matrix has rank at least $k-l+1$.
\end{proof}

%
%

\section{Nuclei}

 Let
$\mathcal{M}=\mathcal{M}(A)=\lbrace L_a \,|\,a\in A\rbrace\subseteq {\rm End}_F(A)$
 be the spread set of an $F$-algebra $A$, where $L_a$ is the left multiplication map in $A$.
We define the \textit{left} and \textit{right idealisers} of $\mathcal{M}$ as
$$I_l(\mathcal{M})=\lbrace \Phi\in {\rm End}_F(A)\,|\, \Phi\mathcal{M}\subseteq\mathcal{M}\rbrace, \text{ respectively, }
I_r(\mathcal{M})=\lbrace \Phi\in {\rm End}_F(A)\,|\, \mathcal{M}\Phi\subseteq\mathcal{M}\rbrace.$$
The \textit{centraliser} of $\mathcal{M}$ is defined as
${\rm Cent}(\mathcal{M})=\lbrace \Phi\in {\rm End}_{F}(A)\,|\, \Phi M=M\Phi \quad \forall M\in \mathcal{M}\rbrace.$ We call
$Z(\mathcal{M})=I_l(\mathcal{M})  \cap{\rm Cent}(\mathcal{M})$ the \textit{center} of $\mathcal{M}$.

\begin{theorem}(cf. \cite[Proposition 5]{Sheekey}  for finite fields) \label{Nucleus of Sheekey algebras}
Let $A$ be a unital division algebra and $\mathcal{M}$ be the spread set of $A$. Let $\mathcal{M}^*$ be the the spread set associated to the opposite algebra $A^{op}$. Then
$${\rm Nuc}_l(A)\cong I_l(\mathcal{M}),\quad {\rm Nuc}_m(A)\cong I_r(\mathcal{M}),\quad
{\rm Nuc}_r(A)\cong {\rm Cent}(\mathcal{M}^*),\quad C(A)\cong Z(\mathcal{M}).$$
\end{theorem}

The proof  from \cite{Sheekey} holds verbatim in our general setting.

 The above results can now be applied to determine the nuclei and center of  the non-unital algebras $S=S_{n,m,l}(\nu,\rho,f)$.

 In the following let $R=D[t;\sigma]$. We use the assumptions on $D$, respectively $K$, and $\sigma$  from Section \ref{sec:division}.

Let $f\in R$ be an irreducible monic polynomial of degree $m$, 
 and
let $h$  be the minimal central left multiple of $f$.  We assume throughout that $f$ is not right-invariant, so that $k>1$. 

\begin{remark}
 The algebras $S_{n,m,l}(0,\rho,f)$ are unital Petit algebras, hence have left nucleus ${\rm Nuc}_m(S)=D$, and their right nucleus $\lbrace g\in R_m\mid fg\in Rf\rbrace$ is the eigenspace of $f$. If $S_{n,m,l}(0,\rho,f)$ is not associative then $\{d\in D\,|\, dg=gd \text{ for all }g\in S\}$ is their center \cite{P66}.
\end{remark}

\begin{theorem} \label{theorem 9I}
Let $R=D[t;\sigma]$ and ${\rm deg}(h)=dmn$. Suppose $l\leq dn/2$, $n>1$ and $lm>2$. Let $S=S_{n,m,l}(\nu,\rho,f)$ and $\mathcal{M}$ be the image of $S$ in ${\rm End}_{E_f}(R/Rf)$, that means the corresponding rank metric code lies in $M_n(E_{\hat{h}})$. If $\nu\neq 0$, we have
\\ (i) $I_l(\mathcal{M})\cong \lbrace g_0\in D \,|\, g_0\nu=\nu\rho( g_0)\rbrace\subset D$ (in particular,
$I_l(\mathcal{M})\cong {\rm Fix}(\rho)$ if $\nu\in C$),
\\ (ii) $I_r(\mathcal{M})\cong {\rm Fix}(\rho^{-1}\circ\sigma^{lm})\subset D$,
\\ (iii) ${\rm Cent}(\mathcal{M})\cong E_{\hat{h}}$, $Z(\mathcal{M})  \cong F'$.
\\ If $\nu=0$, we have
\\ (iv) $I_l(\mathcal{M})\cong D$, $I_r(\mathcal{M})\cong D$,
\\ (v) ${\rm Cent}(\mathcal{M})\cong E_{\hat{h}}$, $Z(\mathcal{M})\cong F$.
\end{theorem}

Much of the proof works identically to the proof of \cite[Theorem 9]{Sheekey}. We sketch the proof to highlight the main differences in this more general case.
The $lm=2$ case has to be considered separately, and we have only been able to solve that for $F=\mathbb{R}$.

\begin{proof}
Let $\mathcal{M}=\{L_a\in {\rm End}_{E_f}(R/Rf)\,|\, a\in P\}$ be the image of $S$ in ${\rm End}_{E_f}(R/Rf)\subset {\rm End}_{F}(R/Rf)$. In the following, we identify each element in $\mathcal{M}$ with the element $g\in S$ that induces it.
 \\
Analogously to the proof of \cite[Theorem 9]{Sheekey}, $\lbrace g\in I_l(\mathcal{M}) \,|\, \text{deg}(g)\leq lm\rbrace= \lbrace g_0\in D \,|\, g_0\nu=\nu\rho(g_0)\rbrace$. If $\nu=0$, then $1\in\mathcal{M}$ so $I_l(\mathcal{M})\subset\mathcal{M}$ so all $g\in I_l(\mathcal{M})$ have degree at most $lm$.\\
Consider $\nu\neq 0$. To check there are no elements $g\in I_l(\mathcal{M})$ of degree higher than $lm$, we follow the approach of \cite[Theorem 9]{Sheekey} and consider $gt \text{ mod } \hat{h}(u^{-1}t^n)$. Recalling $\text{deg}(h)=dm$, we have $h(t)=(u^{-1}t^n)^{dm}+\dots= u^{-dm}(t^n+h'_{dm-1}t^{(dm-1)n}+\dots+h'_0)$ so
$$gt \text{ mod } h(t)=\left(\sum_{i=0}^{dmn-1}g_{i-1}t^i\right)-g_{dmn-1}u^{dm}\left(\sum_{j=0}^{dm-1}h'_jt^{nj}\right).$$ As $g\in I_l(\mathcal{M})$, this implies $gt \text{ mod } h\in \mathcal{M}$, so for all $i\in\lbrace lm+1,\dots, dmn-1\rbrace$, we have
\begin{equation}\label{eqn:nucleus g_i}
g_{i-1}=\begin{cases}
0& \text{for } i\not\equiv 0\text{ mod }n\\
g_{dmn-1}u^{dm}h'_{i/n}& \text{for } i\equiv 0\text{ mod }n
\end{cases}
\end{equation}
where $h'_{i/n}=0$ if $i/n$ is not an integer. We will show that $g_{dmn-1}=0$ and thus $\text{deg}(g)\leq lm-1$. As $lm>2$, this follows verbatim from \cite[Theorem 9]{Sheekey}.

The same holds for $I_r(\mathcal{M})$ following Sheekey's proof with the appropriate amendments made for $D[t;\sigma]$. The results for ${\rm Cent}(\mathcal{M})$ and $Z(\mathcal{M})$ hold verbatim from \cite[Theorem 9]{Sheekey}.
\end{proof}

\begin{corollary}\label{cor:1}
Let $R=D[t;\sigma]$ and $\text{deg}(h)=dmn$. Suppose $n>1$, $m>2$ and $S=S_{n,m,1}(\nu,\rho,f)$ with $\nu\neq 0$ be a division algebra. Then
\\ (i) ${\rm Nuc}_l(S)\cong \lbrace g_0\in D \,|\, g_0\nu=\nu\rho( g_0)\rbrace\subset D$, so in particular
${\rm Nuc}_l(S)= {\rm Fix}(\rho)\subset D$, if $\nu\in C$.
\\ (ii) ${\rm Nuc}_m(S)\cong {\rm Fix}(\rho^{-1}\circ\sigma^{m})\subset D$.
\\ (iii) $ C(S)= {\rm Fix}(\rho)\cap F=F'$.
\\ (iv) ${\rm dim}_{F'}{\rm Nuc}_r(S)={\rm dim}_{F'}(E_{\hat{h}})={\rm deg}(\hat{h})[F:F']=[F:F']dm$.
\end{corollary}

\begin{theorem}\label{theorem 9II}
Let $R=K[t;\sigma]$ and ${\rm deg}(h)=mn$. Suppose $l\leq n/2$, $n>1$ and $lm>2$. Let $S=S_{n,m,l}(\nu,\rho,f)$ with $\nu\neq 0$ and $\mathcal{M}$ be the image of $S$ in  ${\rm End}_{E_f}(R/Rf)$, so that the corresponding rank metric code lies in $M_n(E_{\hat{h}})$. Then
\\ (i) $I_l(\mathcal{M})\cong {\rm Fix}(\rho)\subset K$, $I_r(\mathcal{M})\cong {\rm Fix}(\rho^{-1}\circ\sigma^{lm})\subset K$,
\\ (ii) ${\rm Cent}(\mathcal{M})\cong E_{\hat{h}}$, $Z(\mathcal{M})\cong F'$.
\\
If $\nu=0$, we have
\\ (iii) $I_l(\mathcal{M})\cong K$, $I_r(\mathcal{M})\cong K$,
\\ (iv) ${\rm Cent}(\mathcal{M})\cong E_{\hat{h}}$, $Z(\mathcal{M})\cong F$.
\end{theorem}

Again, the proof is analogous to the one of \cite{Sheekey}, Theorem 9 (it does not use the fact that for finite fields the right nucleus of $S_f$ is $E_{\hat{h}}$, it only uses that $R/Rh$ has center $E_{\hat{h}}$).

\begin{corollary}\label{cor:nuclei K}
Let $R=K[t;\sigma]$ and ${\rm deg}(h)=mn$. Suppose that $n>1$, $m>2$ and that $S=(S(\nu,\rho,f),\circ)$ is a division algebra  with $\nu\neq 0$. Then
\\ (i) ${\rm Nuc}_l(S)= {\rm Fix}(\rho)\subset K$,
\\ (ii) ${\rm Nuc}_m(S)= {\rm Fix}(\rho^{-1}\circ\sigma^m)\subset K$,
\\ (iii) $ C(S)=  {\rm Fix}(\rho)\cap F=F'$.
\\ (iv) ${\rm  dim}_{F'}{\rm Nuc}_r(S)={\rm  dim}_{F'}(E_{\hat{h}})={\rm deg}(\hat{h})[F:F']=[F:F']m$.
\end{corollary}

Theorems \ref{theorem 9I}, \ref{theorem 9II} and Corollaries \ref{cor:1}, \ref{cor:nuclei K} generalize \cite[Theorem 9, Corollary 1]{Sheekey}\label{thm: nuclei} which were proved for semifields.

%
%

\section{Examples of division algebras and an MRD code when $f(t)=t^n-\theta\in K[t;\sigma]$}

\subsection{$K=F(\theta)$}
Let $K=F(\theta)$ be an extension of prime degree $n$. Let $f(t)=t^n-\theta \in K[t;\sigma]$.
We now compute the rank metric code associated to the $F'$-algebra $S_{n,n,1}(\nu,\rho,f)$. Note that
 $f(t) = t^3-\theta \in K[t;\sigma]$ is irreducible if and only if
$\theta\not=\sigma^{2}(z)\sigma(z)z$
for all $z\in K$. If
 $F$ contains a primitive $n$th root of unity, then $f(t)$ is irreducible if and only if
$\theta\not=\sigma^{n-1}(z)\cdots\sigma(z)z$
for all $z\in K$.

We assume that $f$ is irreducible. Define
$h(t)=(t^n-\theta)(t^n-\sigma(\theta))\cdots(t^n-\sigma^{n-1}(\theta))=(t^n)^n+\dots+(-1)^nN_{K/F}(\theta),$
then $h(t)={\rm mclm}(f)$:
as $t^n-\sigma^i(\theta)\in K[t^n]$, the factors of $h(t)$ all commute and $h(t)\in K[t^n]$.
Since
$\sigma(h(t))=(t^n-\sigma(\theta))\cdots(t^n-\sigma^{n-1}(\theta))(t^n-\theta)=h(t),$
  we know that $h(t)\in {\rm Fix}(\sigma)[t]=F[t]$ so $h(t)\in F[t]\cap K[t^n]=F[t^n]=C(R)$. Hence
$h(t)=\hat{h}(t^n)$ with $\hat{h}(x)=x^n+(-1)^nN_{K/F}(\theta)\in F[x]$. Thus
 $f$ divides $h$ both from the left and the right by Lemma \ref{right divisors are left divisors}.

 As $n$ is prime, the minimal central left multiple of $f$ must have degree $n$ in $F[x]$ by Theorem \ref{thm:main2}; thus $h(t)={\rm mclm}(f)$ and hence $\hat{h}(x)=x^n+(-1)^nN_{K/F}(\theta)$ also is an irreducible polynomial in $F[x]$. As a field,
$E_f=\lbrace z+Rf \,|\, z\in F[t^n]\rbrace$ is generated by
$\lbrace 1+Rf,t^n+Rf,t^{2n}+Rf,\dots,t^{n(n-1)}+Rf\rbrace=\lbrace 1+Rf,\theta+Rf,\theta^2+Rf,\dots,\theta^{n-1}+Rf\rbrace$
over $F$.
As $K$ is generated by $\lbrace 1,\theta,\dots, \theta^{n-1}\rbrace$, there is a canonical isomorphism $E_f\longrightarrow K$, $x+Rf\mapsto x.$

It is clear that $\lbrace 1+Rf,t+Rf,\dots,t^{n-1}+Rf\rbrace$ is an $E_f$-basis for $R/Rf$. Let $a=a_0+a_1t+\dots+a_{n-1}t^{n-1}+\nu\rho(a_0)t^n\in S(\nu,\rho,h)$. In order to determine $M_a$, we consider how $L_{a_it^i}$ acts on the basis elements of  $R/Rf$. As left multiplication is distributive, i.e. $L_{a+b}(x)=L_a(x)+L_b(x)$, it follows that
$L_a=\sum_{i=0}^n L_{a_it^i},$ where $a_n=\nu\rho(a_0)$.
For each $i$ we have:
\begin{align*}
L_{a_it^i}(1+Rf)=&a_it^i+Rf=(t^i+Rf)(\sigma^{n-i}(a_i)+Rf)\\
L_{a_it^i}(t+Rf)=&a_it^{i+1}+Rf=(t^{i+1}+Rf)(\sigma^{n-i-1}(a_i)+Rf)\\
\vdots & \qquad \qquad \qquad \qquad \vdots\\
L_{a_it^i}(t^{n-i}+Rf)=&a_it^n+Rf=a_i\theta+Rf=(1+Rf)(a_i\theta+Rf)\\
L_{a_it^i}(t^{n-i+1}+Rf)=&a_it^{n+1}+Rf=a_ix\theta+Rf=(t+Rf)(\sigma(a_i)\theta+Rf)\\
\vdots & \qquad \qquad \qquad \qquad \vdots\\
L_{a_it^i}(t^{n-1}+Rf)=&a_it^{i-1}\theta+Rf=(t^{i-1}+Rf)(\sigma^{n-i+1}(a_i)\theta+Rf).
\end{align*}
Thus the matrix representing $L_{a_it^i}$ is given by
\begin{equation*}M_{a_it^i}= \left( \begin{array}{ccccccccc}
  0                                  &         0                                         &         \cdots      &  0 &     \sigma^{n-i}(a_i)&    0                                    & \cdots       &       0\\
  0                                  &         0                                         &         \cdots      &        0 &0                               &       \sigma^{n-(i+1)}(a_i)   & \cdots      &      0\\
  \vdots                     &                                            &    \ddots         &  &                          &                                            &     \ddots       &       \vdots\\
          0                                   &                                                  &                                      &                \ddots                           &                                                   &                              & & \sigma(a_{m})\\
 a_i\theta                  &                                         &           &               &                       &       \ddots                          &                     & 0\\
          0                                 &         \sigma^{n-1}(a_i)\theta   &                   &           &                    &                                            &                    &  \\
  \vdots                     &                                            &    \ddots       &                 &                         &                                         & \ddots    &    \vdots\\
 0                                        &        0                                        &        \cdots     & \sigma^{n-(i-1)}(a_i)\theta&  0 & 0                         & \cdots    & 0\\
 \end{array} \right).
 \end{equation*}
As $M_a=\sum_{i=0}^nM_{a_it^i}$, we obtain $M_a=(m_{i,j})_{i,j}$ where
\begin{equation*} m_{i,j}=\begin{cases}
\sigma^{n+1-i}(a_0)+\sigma^{n+1-i}(\nu\rho(a_0))\theta &\mbox{for } i=j,\\
\sigma^{n+1-i}(a_{i-j}) &\mbox{for } i>j,\\
\sigma^{n+1-i}(a_{n+i-j})\theta &\mbox{for } i<j.
\end{cases}
\end{equation*}
This yields $\mathcal{C}_{n,n,1}=\lbrace M_a\mid a_k\in K \mbox{ for } k=0,1,\dots,n-1\rbrace\subset M_n(K)$ as the matrix spread set of the $n^2[F:F']$-dimensional $F'$-algebra $S_{n,n,1}(\nu,\rho,f)$.
\\
The algebra associated to this spread set is a division algebra if
$N_{K/F'}(\theta) N_{K/F'}(\nu)\neq 1$ (Theorem \ref{thm: divison algebra K}). In that case the spread set will be an MRD code.
In particular, for $\nu=0$ this condition
 is satisfied for any irreducible $f(t)=t^n-\theta$.
This is the well known result that for  irreducible $f$ the Petit algebra $S_f$ is a division algebra and so are all its isotopes.
 For $n>2$ and $\nu\neq 0$, Corollary \ref{cor:nuclei K} yields
 $$\mathrm{Nuc}_l(S)=\mathrm{Nuc}_m(S)= \mathrm{Fix}(\rho)\subset K,\quad C(S)=F',\quad \mathrm{dim}_{F'}\mathrm{Nuc}_r(S)=[F:F']m.$$

\subsection{Real division algebras of dimension 4}

Over a finite field $F$, all division algebras of dimension 4 over $F$ which have $F$ as their center and a nucleus of dimension 2 over $F$, can be constructed as algebras $S_{n,m,1}(\nu,\rho, f)$ for suitable parameters \cite{Sheekey}. Let us now look at some real division algebras we obtain with our construction. If $\nu=0$, then any choice of an irreducible $f\in \mathbb{C}[t;\can]$
 will yield an algebra isotopic to a real Petit division algebra. If $\nu\neq 0$, any choice of irreducible $f\in \mathbb{C}[t;\can]$ where
  $N_{\mathbb{C}/\mathbb{R}}(a_0)\not = 1/N_{\mathbb{C}/\mathbb{R}}(\nu)$ also yields a division algebra (Theorem \ref{thm: divison algebra K}).

Let $b\in\mathbb{R}$ and $f(t)=t^2-bi\in \mathbb{C}[t;\overline{\phantom{x}}]$.  Then $h(t)=\hat{h}(t^2)$,  $\hat{h}(x)=x^2+b^2\in \mathbb{R}[x]$, is the minimal central left multiple of $f$, as $h(t)=t^4+b^2=(t^2+bi)(t^2-bi).$ For all $b>0$, $f(t)=t^2-bi$ is irreducible in $\mathbb{C}[t;\overline{\phantom{x}}]$.

For every irreducible $f(t)=t^2-bi$, and $\nu\in\mathbb{C}$ such that
$N_{\mathbb{C}/\mathbb{R}}(\nu)\neq \frac{1}{b^2},$ we  obtain a four-dimensional real division algebra $S_{2,2,1}(\nu,\rho,f)$ and an MRD code given by its matrix spread set
$$\mathcal{C}_{2,2,1}=\Big\{ \begin{pmatrix}
z_0+\nu\rho(z_0)bi& z_1bi\\
\overline{z_1}& \overline{z_0}+\overline{\nu\rho(z_0)}bi
\end{pmatrix}
\mid z_0,z_1\in \mathbb{C}
\Big\},$$
where $\rho$ is either the identity or the complex conjugation.

As mentioned in Theorem \ref{thm: nuclei}, \cite[Theorem 9]{Sheekey} uses results to deal with the case when $lm=2$ that are  valid over finite fields, but can be extended to $R=\mathbb{C}[t;\can]$, for instance for $f(t)=t^2-i$:

 \begin{theorem}
Let $f(t)=t^2-i\in \mathbb{C}[t;\can]$. Suppose $S=S_{2,2,1}(\nu,\rho,f)$ is a division algebra for some $\nu\neq 0$ and $\rho\in {\rm Aut}_{\mathbb{R}}(\mathbb{C})$. Then
\\ (i) ${\rm Nuc}_l(S)={\rm Nuc}_m(S)={\rm Fix}(\rho)$,
\\ (ii) $ C(S)=\mathbb{R}$,
\\ (iii)  ${\rm dim}_{\mathbb{R}}({\rm Nuc}_r(S))={\rm dim}_{\mathbb{R}}(\mathbb{R}[t^2])=2$.
\end{theorem}

\begin{proof}
 We have $h(t)=t^4+1\in\mathbb{R}[t^2]$. Suppose $g+Rh\in I_l(\mathcal{M})$ for some $g(t)=g_0+g_1t+g_2t^2+g_3t^3\in R$. Then $ga\in S(\nu,\rho,h)$ for all $a\in S$. Direct and laborious computation yields $g_2=0$, $g_3=-g_1\overline{\nu}$, and $\nu \rho(g_0a_0+g_1\overline{\nu}\overline{a_1})=g_0\nu\rho(a_0)+g_1\overline{a_1}.$ This is satisfied for all $a_0,a_1\in\mathbb{C}$ if and only if $\nu\rho(g_0)=g_0\nu$ and $g_1=\nu\rho(g_1\overline{\nu}).$

 If $\rho=id$, it follows that either $g_1=0$ or $N_{\mathbb{C}/\mathbb{R}}(\nu)=1$; as $S$ is a division algebra, Theorem \ref{Sheekey Theorem 5} (or \cite[Theorem 4]{Sheekey}) forces $g_1=0$ and so $g=g_0$ for some $g_0\in \mathbb{C}$. Thus $I_l(\mathcal{M})=\mathbb{C}$.\\
If $\rho=\can$, then $g_1=\nu^2g_1$ so either $g_1=0$ or $\nu=\pm 1$. As $N_{\mathbb{C}/\mathbb{R}}(\nu)\neq 1$ \cite[Theorem 4]{Sheekey}, this forces $g_1=0$ so $g=g_0$ for some $g_0\in\mathbb{R}$. In this case, $I_l(\mathcal{M})=\mathbb{R}$.

The computations for $I_r(\mathcal{M})$ follow analogously and ${\rm Cent}(\mathcal{M})$ and $Z(\mathcal{M})$ follow from the proof of \cite[Theorem 4]{Sheekey}. We obtain the final result on the nuclei using Theorem \ref{Nucleus of Sheekey algebras} to relate the idealisers and centraliser of $\mathcal{M}$ to the nuclei of the algebra $S$.
\end{proof}

\begin{example}
If $f(t)=t^2-i$ we obtain division algebras $S$ for all $\nu\in\mathbb{C}$ such that $N_{\mathbb{C}/\mathbb{R}}(\nu)\neq 1.$
 If $\nu\neq 0$ and $S$ is a division algebra, then  $ C(S)=  \mathbb{R}$ and
 ${\rm dim}_\mathbb{R}{\rm Nuc}_r(S)= 2.$ Since therefore ${\rm Nuc}_r(S)$ is a two-dimensional division algebra over $\mathbb{R}$,
  ${\rm Nuc}_r(S)$ is an Albert isotope of $\mathbb{C}$ and can be found in the classification in \cite[Theorem 1]{HP}: it must be $ \mathbb{C}$,
 $ \mathbb{C}^{(\can,\can)}$, $ \mathbb{C}^{(1+L(v)\can ,\can)}$, $ \mathbb{C}^{(\can , 1+L(v)\can)}$, or
 $ \mathbb{C}^{(1+L(v)\can ,1+L(w)\can)}$, with $u,v\in \mathbb{C}$ suitably chosen.

 If additionally $\rho=id$, then  ${\rm Nuc}_l(S) ={\rm Nuc}_m(S)= \mathbb{C},$ and  if $\rho=\can$ then
  ${\rm Nuc}_l(S)={\rm Nuc}_m(S)=  \mathbb{R}.$
 Note that the four-dimensional algebras in the first class are all isotopes of nonassociative quaternion algebras.

\end{example}

\section{Constructing algebras and codes using irreducible $f\in R=D[t;\delta]$}\label{sec:elements in N_general}

We now consider the same construction using differential polynomial rings.
Let $C$ a field  of characteristic $p$ and $D$ be a finite-dimensional division algebra with center $C$.
Let $R=D[t;\delta]$, where  $\delta$ is a derivation of $D$, such that $\delta|_C$ is algebraic with minimum polynomial
$g(t)=t^{p^e}+c_1t^{p^{e-1}}+\dots+ c_et\in F[t]$
 of degree $p^e$, with $F=C\cap{\rm Const}(\delta)$. (This includes the special case where $d=1$, i.e. $R=K[t;\delta]$, and $\delta$ is an algebraic derivation with minimum polynomial $g$.)
 Then $g(\delta)=id_{d_0}$ is an inner derivation of $D$.
W.l.o.g. we choose $d_0\in F$, so that $\delta(d_0)=0$ \cite[Lemma 1.5.3]{J96}. Then
$C(D[t;\delta]) = F[x]=\{\sum_{i=0}^{k}a_i(g(t)-d_0)^i\,|\, a_i\in F \}$
 with $x=g(t)-d_0$.
The two-sided $f\in R$ are of the form $f(t)=uc(t)$ with $u\in D$
and $c(t)\in C(R)$ \cite[Theorem 1.1.32]{J96}. All polynomials $f\in R$ are bounded.

For every $f \in R$, the \emph{minimal central left multiple of $f$ in $R$} is the unique  polynomial of minimal degree $h \in C(R)=F[x]$ such that $h = gf$ for some $g \in R$, and such that $h(t)=\hat{h}(g(t)-d_0)$ for some monic $\hat{h}(x) \in F[x]$.
 The bound $f^*$ of $f$ is the unique minimal central left multiple of $f$ up to some scalar.

 From now on let $f \in R=D[t;\delta]$ be a monic irreducible polynomial of degree $m>1$ and let $h(t)=\hat{h}(g(t)-d_0)$ be its minimal central left multiple.
Then $\hat{h}(x)$ is irreducible in $F[x]$ 
 and $h$ generates a maximal two-sided ideal $Rh$ in $R$  \cite[p.~16]{J96}.  We have
$$C(R/Rh)\cong F[x]/F[x] \hat{h}(x)$$
 \cite[Proposition 4]{I99}, and  ${\rm deg}(h)=p^e deg(\hat{h})$. Define $E_{\hat{h}}=F[x]/F[x] \hat{h}(x)$ and
 let $k$ be the number of irreducible factors of $h$ in $R$.

\begin{theorem}\label{thm:main delta} \cite{AO}
${\rm Nuc}_r(S_f)$ is a central division algebra over $E_{\hat{h}}$ of degree $s=d p^e/k$, and
 $$ R/Rh \cong M_k({\rm Nuc}_r(S_f)).$$
 In particular, this means that ${\rm deg}(\hat{h})=\frac{dm}{s}$, ${\rm deg}(h)=km=\frac{d p^e m}{s}$, and
 $$[{\rm Nuc}_r(S_f) :F]= s^2\cdot \frac{dm}{s}=dms.$$
 Moreover, $s$ divides ${\rm gcd}(dm,d p^e)$. If $f$ is not right invariant, then $k>1$ and $s\not=d p^e$.
\end{theorem}

The proof is analogous to the one of Theorem \ref{thm:main2}. In particular, $[S_f:F]=[S_f:C] p^e =d^2 m\cdot p^e$.
Comparing dimensions we obtain again that
$[S_f:{\rm Nuc}_r(S_f)]=k$, and if $f$ is not right-invariant,  $k>1$.

 For each $z(t)=\hat{z}(g(t)-d_0)\in F[g(t)-d_0]$ with $\hat{z}\in F[x]$, we have $z\in Rf$ if and only if $z\in Rh$.
 Let
$$E_f=\lbrace z(t)+Rf \,|\,z(t)=\hat{z}((g(t)-d_0))\in F[(g(t)-d_0)]\rbrace\subset R/Rf.$$
 Together with  the multiplication $(x+Rf)\circ(y+Rf)= (xy)+Rf$
 for all $x,y\in F[(g(t)-d_0)]$, $E_f$ is a field extension of $F$ of degree ${\rm  deg }(\hat{h})$ isomorphic to $E_{\hat{h}}$.
 Let $B={\rm Nuc}_r(S_f)$, then $B$ has degree $s$ over $E_{\hat{h}}$, and $R/Rf$ is a free right $B$-module of dimension $k$ via
$R/Rf\times B\longrightarrow R/Rf$, $(a+Rf)(z+Rf)=az+Rf$.  We assume $f$ is not right-invariant
which yields $k>1$.

For  $\rho\in {\rm Aut}(D)$ define $F'={\rm Fix}(\rho)\cap F$. We assume that $F/F'$ is finite-dimensional. Let $\nu\in D$ and $1\leq l<k=dp^e/s$. Define the set
$S_{p^e,m,l}(\nu,\rho, f)=\lbrace a+Rh\,|\, a\in P\rbrace\subset R/Rh,$
where
$$P=\lbrace a_0+a_1t+\dots+a_{lm-1}t^{lm-1}+\nu\rho(a_0)t^{lm} \,|\, a_i\in D\rbrace\subset D[t;\delta].$$
$S_{p^e,m,l}(\nu,\rho, f)$ is a vector space over $F'$ of dimension $d^2 p^em[F:F']$.
We identify each element of $S_{p^e,m,l}(\nu,\rho,f)$ with a map in ${\rm End}_B(R/Rf)$ as follows:
For each $a\in S_{p^e,m,l}(\nu,\rho,f)$ let $L_a:R/Rf\to R/Rf$ be the left multiplication map  $L_a(b+Rf)=ab+Rf$.
Let $M_a$ be the matrix in $M_k(B)$ representing $L_a$ with respect to a $B$-basis of $R/Rf$, and denote  the image of $S=S_{p^e,m,l}(\nu,\rho,f)$  in $M_k(B)$ by
$$\mathcal{C}_{p^e,m,l}=\lbrace M_a\mid a\in S_{p^e,m,l}(\nu,\rho,f)\rbrace.$$

 For $l=1$, this construction again yields algebras over $F'$:
 define a multiplication on the $F'$-vector space $R_m=\lbrace g\in R \,|\, deg (g)<m\rbrace$
via
$$a(t)\circ b(t)=(a(t)+\nu\rho(a_0)t^m)b(t)\, {\rm mod}_r(f).$$
For $m>1$,
$(R_m,\circ)$ is isomorphic to $S(\nu,\rho,f)=S_{p^e,m,1}(\nu,\rho,f)$. Therefore we also denote
$(R_m,\circ)$ by $S(\nu,\rho,f)=S_{p^e,m,1}(\nu,\rho,f)$.

\begin{example}
Let $R=D[t;\delta]$ and  $f(t)=t+c$ for some $c\in D$. For  $\nu\in D^\times$ and $\rho\in {\rm Aut}(D)$, $S_{p^e,1,1}(\nu,\rho,f)=(D,\circ)$ has the multiplication \begin{align*}
x\circ y =& (x+\nu\rho(x)t)y)\: {\rm mod}_rf
= xy+\nu\rho(x)yt+ \nu\rho(x)\delta(y)\: {\rm mod}_rf \\
= & xy+\nu\rho(x)(\delta(y)-yc)
\end{align*}
for all $x,y\in D$.
\end{example}

\begin{theorem}\label{rank of a polynomial delta}
Let $f\in D[t;\delta]$ be irreducible and ${\rm deg}(h)=km$.  For all $a+Rh\in R/Rh$,
$${\rm dim}_B({\rm im}(L_{M_a}))=
 k^2 - \frac{k}{m} {\rm deg} ({\rm gcrd}(a,\hat{h}(g(t)-d_0))),$$
 $$ {\rm colrank}(M_a)=k-\frac{1}{m} {\rm deg} ({\rm gcrd}(a,\hat{h}(g(t)-d_0)).$$
 In particular, if ${\rm deg}(h)=dmp^e$ then  $M_a\in M_{p^e}(E_{\hat{h}})$  and
 $${\rm  rank}(M_a)= dp^e - \frac{1}{m}{\rm deg}({\rm gcrd}(a,\hat{h}(g(t)-d_0))).$$
\end{theorem}

 Thus $S_{p^e,m,1}(\nu,\rho,f)$
is a division algebra if and only if there are no divisors of $h$ in $S_{p^e,m,1}(\nu,\rho,f)$. More generally for $l>1$, the above result means:

\begin{theorem}
 Suppose that $P$ does not contain any polynomial of degree $lm$, whose irreducible factors are all similar to $f$. Then the set
$S_{p^e,m,l}(\nu,\rho, f)$
defines an $F'$-linear  MRD-code
in $M_k(B)$ with minimum distance $k-l+1$. In particular, if ${\rm deg}(h)=dmp^e$,
then this code is an $F'$-linear MRD-code in $M_{dp^e}(E_{\hat{h}})$ with minimum distance $dp^e-l+1$.
\end{theorem}

\begin{corollary}  \label{thm:division delta}
 Suppose that  $l=1$.
\\ (i) If $a(t)+\nu\rho(a_0)t^m \in P$ is reducible, then $a(t)$ is not a left zero divisor of $(R_m,\circ)$.
\\ (ii) If $\nu=0$ then $(R_m,\circ)$ is a division algebra over $F'$, which for $m\geq 2$ is a Petit algebra.
\\ (iii) If $P$ does not contain any polynomial similar to $f$, then $(R_m,\circ)$ is a division algebra over $F'$.
\end{corollary}

The proofs are all identical to their analogues where $f\in D[t;\sigma]$.

\begin{remark}
One can also use the reduced norm $N$ of the central simple algebra $D(t;\delta)$ in this setting: given a central simple algebra $D$ with a maximal subfield $E$ and $R=D[t;\delta]$,
take the ring of central quotients $D(t;\delta)=\lbrace f/g\mid f\in R, g\in C(R)\rbrace$ of $R$. It has centre $C(D(t;\delta))={\rm Quot}(C(R))=F(x)$, where $x=g(t)=d_0$.
Let $\tilde{\delta}$ be the extension of $\delta$ to $D(x)$ such that $\tilde{\delta}=id_{t\mid D(x)}$. Then $D(t;\delta)$ is a central simple $F(x)$-algebra, more precisely
$D(t;\delta)\cong (D(x),\tilde{\delta},d_0+x),$ i.e. $D(t;\delta)$ is a generalized differential algebra.

Let $N$ be the reduced norm of $D(t;\delta)$. For all $f\in R$, $N(f)\in F[x]$ and $f$ divides $N(f)$.
 Let $\omega:D\to M_d(E)$ be the left regular representation of $D$. For any $f\in R$ of degree $m$, $N(f)=\pm {\rm det}(\omega(a_m))^{p^e}x^{dm}+\dots.$ In particular, $N(f)$ has degree $dm$ \cite{PT}.
As the bound of $f$ has degree $dm$ in $F[x]$, it follows that $N(f)$ is equal to the bound of $f$. Thus if ${\rm deg} (\hat{h})=dm$, we conclude that $\hat{h}(x)=\pm N(f)$.

There is more work to be done, e.g. to determine the constant term of $N(f(t))$. This may lead to criteria on how to obtain division algebras using our construction.
Additionally, the nuclei of both the algebras and the codes need to be calculated.
For instance, consider the special case where $d=1$, i.e. $R=K[t;\delta]$ for some field extension $K/F$.
If $f(t)=a_0+a_1t+\dots+a_mt^m\in R=K[t;\delta]$ has degree $m$, then
 $N(f(t))=(-1)^{m(p^e-1)} a_m ^{p^e}x^m+\dots$ \cite[Thm 18(ii)]{PT}
Thus
$$N(f(t))=a_m^{p^e}x^m+\dots$$
To find the constant term of $N(f)$ is difficult.
It is possible to compute special cases though,
e.g.  for $f(t)=g(t)+a\in K[t;\delta]$, $N(f(t))=(x+a)^{p^e}$  \cite{PT}.
\end{remark}

\emph{Acknowledgement:} We would like to thank J. Sheekey for several helpful discussions on the subject, and the referee, whose comments greatly improved our paper.


\end{document}